\documentclass[a4paper,11pt]{article}
\usepackage{amsmath, amssymb}
\usepackage{amsthm}
\usepackage[OT1]{fontenc}

\usepackage[margin=1.1in]{geometry}

\newtheorem{theorem}{Theorem}[section]
\newtheorem{definition}[theorem]{Definition}
\newtheorem{lemma}[theorem]{Lemma}

\newtheorem{remark}[theorem]{Remark}
\newtheorem{example}[theorem]{Example}
\numberwithin{equation}{section}
\title{ Weak Necessary and Sufficient Stochastic Maximum Principle for Markovian Regime-Switching Diffusion Models}

\date{}
\author{Yusong Li\footnote{Department of Mathematics, Imperial College, London SW7 2BZ, UK.  Email: y.li11@imperial.ac.uk}\; and Harry Zheng\footnote{Department of Mathematics, Imperial College, London SW7 2BZ, UK.  Email: h.zheng@imperial.ac.uk}}

\begin{document}

\maketitle

\begin{abstract}
In this paper we prove a weak necessary and sufficient maximum principle for Markovian regime switching stochastic optimal control problems. Instead of insisting on the maximum condition of the Hamiltonian, we show that $ 0 $ belongs to the sum of Clarke's generalized gradient of the Hamiltonian and Clarke's normal cone of the control
constraint set at the optimal control. Under a joint  concavity condition on the Hamiltonian and a convexity condition on the terminal objective function, the necessary condition becomes sufficient. We give four examples  to demonstrate  the weak
stochastic maximum principle. 
\end{abstract}



\section{Introduction}
There has been extensive research in the stochastic control theory. Two principal and most commonly used methods in solving stochastic optimal control problems are the dynamic programming principle and the stochastic maximum principle (SMP). 
The books by Fleming-Rishel \cite{flemingrishel:detestochcontrol},  Fleming-Soner \cite{flemingsoner:controledmarkprocess}, and Yong-Zhou  \cite{yong.zhou:stochasticcontrols} provide excellent expositions and rigorous treatment of the subject of the dynamic programming principle in the optimal deterministic and stochastic control theory. 

Many people have made great contributions in the research of the SMP.
Kushner \cite{kushner:fixedtimecontrol, kushner:nececond} is the first to study 
the necessary SMP. Haussmann \cite{haussmann:SMPdiffusion}, Bensoussan \cite{bensoussan:lecturestochcontrol} and Bismut \cite{bismut:conjugateconvex, bismut:linearquarcontrol, bismut:dualityoptcontrol} extend 
Kushner's SMP to more general 
  stochastic control problems with control-free diffusion coefficients. 
Peng \cite{peng:SMP}  applies the second order spike variation technique to derive the necessary SMP to  stochastic control problems with controlled diffusion coefficients. Zhou \cite{zhou:unifiedtreatment} simplifies Peng's proof. Cadenillas-Karatzas \cite{karatzas:SMPrandomcoeff} extends Peng's SMP  to systems with random coefficients
and Tang-Li \cite{tangli:SMPjump}  with jump diffusions. 
Bismut \cite{bismut:dualityoptcontrol} is the first to investigate the sufficient SMP. 
 Zhou \cite{zhou:sufficientSMPwithcontroldiffus} proves that Peng's SMP is also sufficient in the presence of certain convexity condition. Framstad-{\O}ksendal-Sulem \cite{oksendal:SSMP} extends the sufficient SMP to systems with jump diffusion, Donnelly \cite{donnelly:SMPregimeswitching}  with Markovian regime-switching diffusion
and,  most recently, Zhang-Elliott-Siu \cite{zhangelliottsiu:SMPregimeswitchingjump} with Markovian regime-switching jump diffusion.

Briefly speaking, the necessary SMP states that any optimal control along with the optimal state trajectory must solve a system of forward-backward SDEs (stochastic differential equations) plus a maximum condition of the optimal control on the Hamiltonian. The necessary condition together with certain concavity  conditions on the Hamiltonian give the sufficient condition of optimality. The major difficulty of generalizing the classical Pontryagin's maximum principle to a stochastic control problem with controlled diffusion term is that, in some cases,  the Hamiltonian  is a convex function of the control variable  and achieves the minimum at the optimal control (see \cite[Example 3.3.1]{yong.zhou:stochasticcontrols}). One of the major contributions of Peng's SMP is the introduction of the generalized Hamiltonian  and the second order adjoint stochastic processes. In those cases where the Hamiltonian is convex, it is the second order term that turns the generalized Hamiltonian  to a concave function which achieves the maximum at the optimal control. The generalized Hamiltonian and the second order adjoint equation are introduced
to preserve the maximum condition of Pontryagin's maximum principle.

However, the second order terms also pose problems. Firstly,  one has to  assume that all functions involved are twice continuously differentiable in the state variable in order to use the second order variation, which limits the scope of problems applicable to the theorem. Secondly, one has to solve
the associated second order adjoint backward stochastic differential equation (BSDE) with the dimensionality equal to the square of that of its first order counterpart,
which makes the problem more difficult to solve, at least numerically. Lastly, one can not get the sufficient condition by enhancing the necessary condition
with some joint concavity condition to the generalized Hamiltonian and instead
one has to add some joint concavity condition to the Hamiltonian (compare
\cite[Theorem 3.3.2]{yong.zhou:stochasticcontrols} and \cite[Theorem 3.5.2]{yong.zhou:stochasticcontrols}), which illustrates
that the necessary SMP is not completely compatible
with the sufficient SMP. This motivates us to relax
the requirement of the maximality of  the Hamiltonian at the optimal control and to seek a weak but compatible necessary and sufficient SMP.

In this paper we assume that the control constraint set is a closed convex set. The second order adjoint processes can also be dropped in \cite{peng:SMP}, see \cite{yong.zhou:stochasticcontrols}, under the differentiability conditions for state and control variables. However, the philosophy of this paper is different from that of \cite{peng:SMP} in the sense that we do not try to preseve Pontryagin's maximum principle but instead try to find all stationary points of the Hamiltonian, which may open the way for new results when the control constraint set is  nonconvex.

The main contribution of this paper is that we prove a weak version of the
necessary and sufficient SMP for Markovian regime switching diffusion stochastic optimal control problems. Instead of insisting on the Hamiltonian to achieve the maximum at the optimal control, which is in general impossible,
we relax the necessary condition by only requiring the optimal control
to be a stationary point of the Hamiltonian. Specifically, 
we prove that 0 belongs to the sum of Clarke's generalized gradient of the Hamiltonian
and Clarke's normal cone of the control constraint set at the optimal control
almost surely almost everywhere. Under the joint concavity condition on the
Hamiltonian and the convexity condition on the terminal objective function, the necessary condition becomes the sufficient condition. 

The advantage of the weak SMP is the following. 
Firstly, the second order differentiability of the coefficients and the objective functions in the state variable is not required as the weak SMP does not have any second order terms.
 Secondly, the differentiability of the coefficients and the objective functions in the control variable is not required as the weak SMP uses Clarke's generalized gradients  to describe the optimal control. Thirdly,
the dimensionality of the BSDE is much reduced as the second order adjoint
process is not involved. 
Lastly, the necessary condition and the sufficient
condition are compatible with each other in the sense that the necessary condition provides a stationary point
while the sufficient condition confirms its optimality, which
is in the same spirit as the necessary and sufficient conditions in the  finite dimensional optimization.

The rest of the paper is organized as follows. Section 2 introduces the notations,
 the formulation of the regime switching stochastic control problem and the basic assumptions. 
Section 3 states  the  main theorems of the paper, the weak necessary SMP  (Theorem \ref{WNSMP}) and the weak sufficient SMP  (Theorem \ref{WSSMP}).
Section 4 gives four examples to demonstrate the usefulness of the weak SMP in solving regime switching stochastic control problems, including nonsmooth noncave case and regime-switching noncave case.  Section 5 establishes some useful preliminary results on Clarke's generalized gradient and normal cone, Markovian regime switching SDE and BSDE, moment estimates, Lipschitz property, Taylor expansion and duality analysis. Section 6  proves the main theorems. Section 7 concludes.
The appendix gives the proof of Theorem \ref{RSBSDEtheorem} (existence and uniqueness of the solution to a regime switching BSDE) for completeness.

\section{Problem Formulation}
In this section, we formulate the stochastic control problem in a regime switching diffusion model and introduce some assumptions. Here we adopt the model in \cite{donnelly:SMPregimeswitching}

Let $ \left(\Omega,\mathcal{F},\mathbb{P}\right) $ be a complete probability space with a $ \mathbb{P} $ complete right continuous filtration. Let the previsible $\sigma$-algebra on $ \Omega\times [0,T] $ associated with the filtration $ \left\{ \mathcal{F}_t:t\in[0,T] \right\} $, denoted by $ \mathcal{P}^{\star} $, be the smallest $ \sigma $-algebra on $ \Omega\times[0,T] $ such that every $ \left\{ \mathcal{F}_t \right\} $-adapted stochastic process which is left continuous with right limit is $ \mathcal{P}^\star $ measurable. A stochastic process $ X $ is previsible, written as $ X\in\mathcal{P}^\star $,  provided it is $ \mathcal{P}^\star $ measurable.

Let $ W(\cdot) $ be  an $ m $-dimensional standard Brownian motion and $ \alpha(\cdot) $ a continuous time finite state observable Markov chain, which are independent of each other. $ \left\{ \mathcal{F}_t \right\} $ is the natural filtration generated by $ W $ and $ \alpha $, completed with all $ \mathbb{P} $-null sets, denoted by
\begin{align*}
\mathcal{F}_t=\sigma\left[W(s):0 \leq s \leq t\right]\bigvee\sigma\left[\alpha(s):0\leq s\leq t\right]\bigvee\mathcal{N},
\end{align*}
where $ \mathcal{N} $ denotes the totality of $ \mathbb{P} $-null sets. 

Let the Markov chain take values in the state space $ I=\left\{ 1,2,\cdots,d-1,d \right\} $ and start from initial state $ i_0\in I $ with a $ d\times d $ generator matrix $ \mathcal{Q}=\left\{ q_{ij} \right\}_{i,j=1}^d $. For each pair of distinct states $ \left( i,j \right) $, define the counting process $ \left[ Q_{ij} \right] :\Omega\times[0,T]\rightarrow\mathbb{N} $ by
\begin{align*}
[Q_{ij}](\omega,t):=\sum_{0 < s\leq t}\mathcal{X}\left[ \alpha(s-)=i \right](\omega)\mathcal{X}[\alpha(s)=j](\omega),\forall t\in[0,T],
\end{align*}
and the compensator process $ \langle Q_{ij} \rangle:\Omega\times[0,T]\rightarrow[0,+\infty )$ by
\begin{align*}
\langle Q_{ij}\rangle(\omega,t):=q_{ij}\int_0^t\mathcal{X}\left[ \alpha(s-)=i \right](\omega)ds,\forall t\in[0,T],
\end{align*}
where $ \mathcal{X} $ is an indicator function. The processes 
$$ Q_{ij}(\omega,t):=[Q_{ij}](\omega,t)-\langle Q_{ij}\rangle(\omega,t)$$
 is a purely discontinuous square-integrable martingale with initial value zero (\cite[Lemma IV.21.12]{rw:diffusions}).

Consider a stochastic control model where the state of the system is governed by a controlled Markovian regime-switching SDE:
\begin{equation}\label{stateeqn}
\left\{ \begin{array}{cl}
dx(t)= & b(t,x(t),u(t),\alpha(t-))dt+\sigma(t,x(t),u(t),\alpha(t-))dW(t)\\
x(0)= & x_0\in\mathbb{R}^n, \alpha(0)=i_0\in I,
\end{array}
\right.
\end{equation}
where $ u(\cdot)$ is a $ \mathbb{R}^k $ valued previsible process, $ T>0 $ is a fixed finite time horizon, $ b:[0,T]\times\mathbb{R}^n\times\mathbb{R}^k\times I \rightarrow \mathbb{R}^n$ and $\sigma:[0,T]\times\mathbb{R}^n\times\mathbb{R}^k\times I \rightarrow \mathbb{R}^{n\times m}$ are given continuous functions satisfying the following assumptions:
\begin{enumerate}
\item[\textbf{(A1)}] The maps $ b$ and $ \sigma $ are measurable, and there exist constant $ K>0 $ such that for $ \varphi=b \text{ and } \sigma $, we have
\begin{align*}
\left\{
\begin{array}{ll}
\vert \varphi(t,x,u,i)-\varphi(t,\hat{x},\hat{u},i) \vert \leq K\left( \vert x-\hat{x} \vert + \vert u-\hat{u} \vert\right)\\
\forall t\in[0,T]; i \in I ; x,\hat{x}\in\mathbb{R}^n ; u,\hat{u}\in\mathbb{R}^k,\\
\vert \varphi(t,0,0,i) \vert <K, \ \forall t\in[0,T], \forall i\in I.
\end{array}
\right.
\end{align*}

\item[\textbf{(A2)}] The maps $ b $ and $\sigma $ are $ C^1 $ in $ x $ and there exist a constant $ L>0 $ and a modulus of continuity $ \bar{\omega} : [0,+\infty)\rightarrow [0,+\infty)$ such that
\begin{align*}
\left\{
\begin{array}{ll}
\vert \varphi_x(t,x,u,i)-\varphi_x(t,\hat{x},\hat{u},i) \vert \leq L\vert x-\hat{x} \vert + \bar{\omega}(d(u,\bar{u}))\\
 \forall t\in[0,T]; i \in I ; x,\hat{x}\in\mathbb{R}^n ; u,\hat{u}\in\mathbb{R}^k,
 \end{array}
 \right.
\end{align*}
where $\varphi_x(t,x,u,i)$ is the partial derivative of $\varphi$ with respect to $x$ at the point $(t,x,u,i)$. 
\end{enumerate}

Consider the cost functional 
\begin{equation}\label{costfun}
J(u)= E\left[ \int_0^T f(t,x(t),u(t),\alpha(t))dt + h(x(T),\alpha(T)) \right],
\end{equation}
where $ f:[0,T]\times\mathbb{R}^n\times\mathbb{R}^k\times I \rightarrow \mathbb{R} $ and $ h:\mathbb{R}^n\times I \rightarrow \mathbb{R} $ are given functions satisfying the following assumptions:
\begin{enumerate}
\item[\textbf{(A3)}] The maps $ f $ and $ h $ are measurable and there exist constants $ K_1, K_2 \geq 0 $ such that
\begin{align*}
\left\{
\begin{array}{ll}
\vert f(t,x,u,i)-f(t,x,\hat{u},i) \vert \leq \left[ K_1+K_2(\vert x \vert + \vert u \vert + \vert \hat{u} \vert) \right]\vert u-\hat{u} \vert,\\
\vert f(t,0,0,i) \vert+\vert h(0,i) \vert <K_1, \ \forall t\in[0,T], \forall i\in I.
\end{array}
\right.
\end{align*}
\item[\textbf{(A4)}] The maps $ f $ and $ h $ are $ C^1 $ in $ x $ and  there exist a constant $ L>0 $ and a modulus of continuity $ \bar{\omega}:[0,+\infty)\rightarrow[0,+\infty) $ such that for $ \varphi=f \text{ and } h $, we have
\begin{align*}
\left\{
\begin{array}{ll}
\vert \varphi_x(t,x,u,i)-\varphi_x(t,\hat{x},\hat{u},i) \vert \leq L\vert x-\hat{x} \vert + \bar{\omega}(d(u,\bar{u})),\\
\forall t\in[0,T]; i \in I ; x,\hat{x}\in\mathbb{R}^n ; u,\hat{u}\in\mathbb{R}^k,\\
\vert \varphi_x(t,0,0,i) \vert \leq L, \forall t\in[0,T], i\in I.
\end{array}
\right.
\end{align*}
\end{enumerate}
\begin{remark}\rm
Assumptions \textbf{(A3)} and \textbf{(A4)} together cover many cases, including all quadratic functions in $ x $ and $ u $. For instance, if $ f $ is Lipschitz in $ u $, then $ K_2=0 $. On the other hand, if $ f $ is differentiable with respect to $ u $ and $ f_u $ satisfies a linear growth condition in $ u $, then $ K_2 $ is a positive constant.
\end{remark}

Consider a measure space $ (S, \mathcal{P}^\star, \mu) $, where $ S=\Omega \times [0,T] $ and $ \mu = \mathbb{P}\times Leb $. Define $ L^p(S;\mathbb{R}^q) \text{ for } p,q\in\mathbb{N}^+ $ to be the Banach space of $ \mathbb{R}^q $ valued $ \mathcal{P}^\star $ measurable functions $ f:\Omega\times [0,T]\rightarrow \mathbb{R}^q $ such that
\begin{equation}
\|f\|:=\left( \int_0^T E\vert f(t) \vert^p dt \right)^\frac{1}{p} < \infty.
\end{equation}
Similarly, define $ L^p_{\mathcal{F}}(S;\mathbb{R}^q) \text{ for } p,q\in\mathbb{N}^+ $ to be the space of $ \mathbb{R}^q $ valued $ \mathcal{F}_t $ progressively measurable $ p $th order integrable processes.

According to Theorem \ref{exisuniqSDE}, under assumption \textbf{(A1)}, for any $ u\in L^4(S;\mathbb{R}^k) $, the state equation \eqref{stateeqn} admits a unique solution and the cost functional \eqref{costfun} is well defined. A control is called admissible if it is valued in $ U $, a non-empty closed convex subset of $ \mathbb{R}^k $ and $ u\in L^4(S;\mathbb{R}^k) $. Denoted by $ \mathcal{U}_{ad} $ the set of admissible controls. In the case that $ x $ is a solution of \eqref{stateeqn} corresponding to an admissible control $ u\in \mathcal{U}_{ad} $, we call $ (x,u) $ an admissible pair and $ x $ an admissible state process.

Our optimal control problem can be stated as follows\\
\\
\textbf{Problem (S)} Minimize \eqref{costfun} over $ \mathcal{U}_{ad} $.\\
\\
Any $ \bar{u}\in\mathcal{U}_{ad} $ satisfying 
\begin{align*}
J(\bar{u})=\inf_{u\in\mathcal{U}_{ad}}J(u)
\end{align*}
is called an \textit{optimal control}. The corresponding $ \bar{x} $ and $ (\bar{x},\bar{u}) $ are called an \textit{optimal state process} and \textit{optimal pair}, respectively.

\section{Weak Stochastic Maximum Principle}
In this section we state the weak necessary and sufficient stochastic maximum principle in the regime-switching diffusion model.

The Hamiltonian $ H:[0,T]\times\mathbb{R}^n\times\mathbb{R}^k\times I\times\mathbb{R}^n\times\mathbb{R}^{n\times m} \rightarrow \mathbb{R}$ for the stochastic control problem \eqref{stateeqn} and \eqref{costfun} is defined by:
\begin{align}\label{Hamiltonian}
\begin{split}
H(t,x,u,i,p,q):=&-f(t,x,u,i)+b^\intercal(t,x,u,i)p+tr(\sigma^\intercal(t,x,u,i)q).
\end{split}
\end{align}
Given an admissible pair $ (x,u) $, the adjoint equation in the unknown adapted processes $ p(t)\in\mathbb{R}^n, q(t)\in\mathbb{R}^{n\times m} $ and $ s(t)=(s^{(1)}(t),\cdots,s^{(n)}(t)) $, where $ s^{(l)}(t)\in\mathbb{R}^{d\times d} $ for $ l=1,\cdots,n $, is the following regime-switching BSDE:
\begin{equation}\label{ajointeqn}
\left\{
\begin{array}{ll}
dp(t)=&-H_x(t,x(t),u(t),\alpha(t-),p(t),q(t))dt+q(t)dW(t)+s(t)\bullet dQ(t)\\
p(T)=&-h_x(x(T),\alpha(T)),
\end{array}
\right.
\end{equation}
where
\begin{align*}
s(t)\bullet dQ(t)\equiv\left( \sum_{j\neq i}s^{(1)}_{ij}(t)dQ_{ij}(t),\cdots , \sum_{j\neq i}s^{(n)}_{ij}(t)dQ_{ij}(t)\right)^\intercal.
\end{align*}

By Theorem \ref{RSBSDEtheorem}, we claim that under assumptions \textbf{(A1)-(A4)}, for any $ (x,u)\in L^2_{\mathcal{F}}(S;\mathbb{R}^n)\times L^4(S;\mathbb{R}^k) $, \eqref{ajointeqn} admits a unique solution $ \{ (p(t),q(t),s(t))\vert t\in[0,T]\} $ in the sense of Definition \ref{RSBSDEdefn}. If $ (\bar{x},\bar{u}) $ is an optimal (resp. admissible) pair and $ (\bar{p},\bar{q},\bar{s}) $ is the adapted solution of \eqref{ajointeqn}, then $ (\bar{x},\bar{u}, \bar{p},\bar{q},\bar{s}) $ is called an optimal (resp. admissible) 5-tuple.

We can now state the main results of the paper.

\begin{theorem}\label{WNSMP}
(Weak Necessary SMP with Regime-Switching) Let assumptions \textbf{(A1)-(A4)} hold. Let $ (\bar{x},\bar{u}) $ be an optimal pair of \textbf{Problem (S)}. Then there exists stochastic process $ (\bar{p},\bar{q},\bar{s}) $ which is an adapted solution to \eqref{ajointeqn}, such that 
\begin{equation}\label{neceSMPcond}
0\in\partial_u (-H)(t,\bar{x}(t),\bar{u}(t),\alpha(t-),\bar{p}(t),\bar{q}(t))+N_U(\bar{u}(t)), \mbox{a.e. }t\in[0,T], \mathbb{P}\mbox{-a.s.},
\end{equation} 
where $ \partial_u (-H)(t,\bar{x}(t),\bar{u}(t),\alpha(t-),\bar{p}(t),\bar{q}(t))$ 
is Clarke's generalized gradient of $-H$ with respect to variable $u$ at point 
$(t,\bar{x}(t),\bar{u}(t),\alpha(t-),\bar{p}(t),\bar{q}(t))$ and $N_U(\bar{u}(t))$ is Clarke's  normal cone of $U$ at point $\bar{u}(t)$ (see Subsection \ref{ClarkeIntro} for details).
\end{theorem}
\begin{theorem}\label{WSSMP}
(Weak Sufficient SMP with Regime-Switching) Let assumptions \textbf{(A1)-(A4)} hold and let $ (\bar{x},\bar{u},\bar{p},\bar{q},\bar{s}) $ be an admissible 5-tuple satisfying \eqref{neceSMPcond}. Suppose further that $ h(\cdot,\alpha(T)) $ is convex and the Hamiltonian $ H(t,\cdot,\cdot,\alpha(t-),\bar{p}(t),\bar{q}(t)) $ is concave for all $ t\in[0,T] $ a.s. Then $ (\bar{x},\bar{u}) $ is an optimal pair for \textbf{Problem (S)}.
\end{theorem}
\begin{remark}\rm
In the special case where $ \mathcal{F}_t=\sigma[W(s):0\leq s\leq t]\bigvee \mathcal{N} $, i.e., the randomness of the system is generated only by the Brownian motion, the Hamiltonian (\ref{Hamiltonian}) and all other functions are free of index $i$ or Markov chain processs value $\alpha(t-)$. The adjoint equation (\ref{ajointeqn}) is a pure Brownian BSDE (no $s(t)\bullet dQ(t)$ term). The weak SMP remains the same as Theorem \ref{WNSMP} and \ref{WSSMP}, but only involves the $ 4 $-tuple $ (\bar{x},\bar{u},\bar{p},\bar{q}) $.
\end{remark}

\section{Examples}
In this section, we present four examples to demonstrate our main theorems. 
\subsection{Examples: Weak SMP without Regime-Switching}
In this subsection, we consider two examples from \cite{yong.zhou:stochasticcontrols} and derive the same results as  those in  \cite{yong.zhou:stochasticcontrols} using Theorem \ref{WNSMP} and Theorem \ref{WSSMP}. A key property to use in our approach is the adaptedness of the adjoint process.

\begin{example}(Concave Hamiltonian)\label{example1}
{\rm
Consider the following stochastic control problem
\cite[Example 3.5.3]{yong.zhou:stochasticcontrols}:
\begin{equation}\label{example1stateqn}
\left\{
\begin{array}{cl}
& dx(t)=u(t)dW(t), t\in[0,1]\\
& x(0)=0
\end{array}\right.
\end{equation}
with the control constraint set $ U=[0,1] $ and the cost functional 
\begin{align*}
J(u) = E\left\{ -\int_0^1 u(t)dt +\frac{1}{2}x(1)^2 \right\}.
\end{align*}
Suppose $ (\bar{x},\bar{u}) $ is an optimal pair, then the corresponding adjoint equation is
\begin{equation}\label{example1adjointqn}
\left\{
\begin{array}{cl}
& d\bar{p}(t)=\bar{q}(t)dW(t), t\in[0,1]\\
& \bar{p}(1)=-\bar{x}(1).
\end{array}\right.
\end{equation}
Using \eqref{example1stateqn} and \eqref{example1adjointqn} and via a simple calculation we obtain
\begin{align*}
\bar{p}(t)=-\int_0^t\bar{u}(s)dW(s)-\int_t^1\left( \bar{u}(s)+\bar{q}(s) \right)dW(s).
\end{align*}
Since the adjoint process $ \bar{p}(t) $ is adapted to the filtration $ \mathcal{F}_t $ , we must have
\begin{equation}\label{adapted1}
\bar{u}(t)+\bar{q}(t)=0 \textit{ for all } t\in[0,1], \mathbb{P}\textit{-a.s.}
\end{equation}
The corresponding Hamiltonian is 
\begin{align*}
H(t,x,u,\bar{p}(t),\bar{q}(t)) = \bar{q}(t)u + u.
\end{align*}
Since the problem satisfies \textbf{(A1)}-\textbf{(A4)}, by Theorem \ref{WNSMP} and 
(\ref{neceSMPcond}),  we have 
\begin{align*}
0 &\in -(\bar{q}(t)+1)+N_{[0,1]}(\bar{u}(t)) \textit{ for all } t\in[0,1], \mathbb{P}\textit{-a.s}.
\end{align*}
Consequently, on any nonzero measurable set $ E\in S=\Omega \times [0,1] $, we can only have the following three cases: 
\begin{description}
\item[Case 1]: $ 0<\bar{u}(t)<1 \Longrightarrow N_{[0,1]}\left( \bar{u}(t) \right)=\{0\} \textit{ and } \bar{q}(t)=-1 $.
\item[Case 2]: $ \bar{u}(t)=0 \Longrightarrow N_{[0,1]}\left( \bar{u}(t) \right)= (-\infty,0] \textit{ and } \bar{q}(t)+1 \leq 0 $.
\item[Case 3]: $ \bar{u}(t)=1 \Longrightarrow N_{[0,1]}\left( \bar{u}(t) \right)= [0,+\infty) \textit{ and } \bar{q}(t)+1 \geq 0 $.
\end{description}
Suppose Case 1 or Case 2 is true, then $ \bar{u}(t)+\bar{q}(t)\leq \bar{u}(t)-1< 0 $ for some nonzero measurable set $ E\in S $, contradiction to (\ref{adapted1}). Hence, we have $ \bar{u}(t)=1 $ for every $t\in[0,1], \mathbb{P}$-a.s.  and $ \bar{x}(t)=W(t) $ and $ (\bar{p}(t),\bar{q(t)})=(-W(t),-1) $ for $ t\in [0,1] $. Since $(x,u)\mapsto H(t,x,u, \bar{p}(t),\bar{q}(t)) = -u+u=0 $ is concave and $x\mapsto h(x)=\frac{1}{2}x^2 $ is convex, we conclude that $ \bar{u}(t)=1 $ is the optimal control using Theorem \ref{WSSMP}.
}\end{example}

\begin{example}(Nonconcave nonsmooth Hamiltonian)\label{nonsmoothexample}
{\rm
Consider the following stochastic control problem
\begin{equation}\label{nonsmoothstateqn}
\left\{
\begin{array}{cl}
& dx(t)=\dfrac{1}{2}\vert u(t)\vert dW(t), t\in[0,1]\\
& x(0)=0
\end{array}\right.
\end{equation}
with the control constraint set $ U=[-1,1] $ and the cost functional 
\begin{align*}
J(u) = E\left\{\int_0^1 [x(t)^2-\frac{1}{2}u(t)^2]dt + x(1)^2\right\}.
\end{align*}
Suppose $ (\bar{x},\bar{u}) $ is an optimal pair, then the corresponding adjoint equation is 
\begin{equation}\label{nonsmoothadjointqn}
\left\{
\begin{array}{cl}
& d\bar{p}(t)=2\bar{x}(t)dt+\bar{q}(t)dW(t), t\in[0,1]\\
& \bar{p}(1)=-2\bar{x}(1).
\end{array}\right.
\end{equation}
Using \eqref{nonsmoothstateqn}, \eqref{nonsmoothadjointqn} and via a simple calculation, we obtain
\begin{align*}
\bar{p}(t)=-\int_0^t(2-t)\vert\bar{u}(s)\vert dW(s)-\int_t^1((2-s)\vert\bar{u}(s)\vert +\bar{q}(s))dW(s).
\end{align*}
Since the adjoint process $ \bar{p}(t) $ is adapted to the filtration $ \mathcal{F}_t $, we must have 
\begin{equation} \label{4.9}
(2-t)\vert\bar{u}(t)\vert +\bar{q}(t)=0 \textit{ for all } t\in[0,1], \mathbb{P}\textit{-a.s.}
\end{equation}
The corresponding Hamiltonian is 
\begin{align*}
H(t,x,u,\bar{p}(t),\bar{q}(t))=\dfrac{1}{2}\bar{q}(t)\vert u\vert -x^2+\frac{1}{2}u^2.
\end{align*}
Since the problem satisfies assumptions \textbf{(A1)}-\textbf{(A4)}, by Theorem \ref{WNSMP} and (\ref{neceSMPcond}), we have
\begin{align}\label{nonsmoothcontra}
0 & \in  \partial_u\left( x(t)^2-\frac{1}{2}q(t)\vert u(t) \vert -\frac{1}{2}u(t)^2 \right) + N_{[-1,1]}(\bar{u}(t))\textit{ for all } t\in[0,1], \mathbb{P}\textit{-a.s}.
\end{align}
Consequently, on any nonzero measurable set $ E\in S $, we can only have the following five cases:
\begin{description}
\item[Case 1] $\bar{u}(t)=1 \Longrightarrow 0\in \left\{-\dfrac{1}{2}q(t)-1\right\}+[0,+\infty)$ which is compatible with the adaptedness condition (\ref{4.9}) $ \bar{q}(t)=t-2 $.
\item[Case 2] $\bar{u}(t)=-1 \Longrightarrow 0\in \left\{\dfrac{1}{2}q(t)+1\right\}+(-\infty,0]$ which is compatible with  (\ref{4.9}) $ \bar{q}(t)=t-2 $.
\item[Case 3] $\bar{u}(t)=0 \Longrightarrow 0\in \left[\dfrac{1}{2}q(t),-\dfrac{1}{2}q(t)\right]+\{0\}$ 
which is compatible with  (\ref{4.9}) $ \bar{q}(t)=0 $.
\item[Case 4] $\bar{u}(t)\in(0,1) \Longrightarrow 0\in \left\{-\dfrac{1}{2}q(t)-\bar{u}(t)\right\}+\{0\}$ which gives $q(t)=-2\bar{u}(t)<0$, a contradiction to (\ref{4.9}) $ \bar{q}(t)=(t-2)\bar{u}(t)>0$. 
\item[Case 5] $\bar{u}(t)\in(-1,0) \Longrightarrow 0\in \left\{\dfrac{1}{2}q(t)-\bar{u}(t)\right\}+\{0\}$ which gives  $q(t)=2\bar{u}(t)$, a contradiction to (\ref{4.9})
$ \bar{q}(t)=(2-t)\bar{u}(t) $.
\end{description}
Hence, the set of optimal candidates from Weak Necessary SMP consists of all the progressively measurable processes valued in the set $ \{-1,0,1\} $. However,  since the Hamiltonian is not concave, Theorem \ref{WSSMP} cannot be applied.  Substituting $ x(t)=\int_0^t \frac{1}{2}\vert u(s)\vert dW(s) $ into the cost functional and by simple calculations, we obtain
\begin{align*}
J(u) = -\dfrac{1}{4}E\int_0^1 t \vert u(t)\vert ^2 dt.
\end{align*}
Hence $J(u)$ reaches the minimum at $|\bar u(t)|=1$ a.s. for all $t$, which implies   there are infinitely many optimal controls  with any measurable combination of $ 1 $ and $ -1 $.
The optimal state process  is $\bar x(t)={1\over 2}W(t)$ and the adjoint processes are
$\bar p(t)=(t-2)W(t)$ and $\bar q(t)=t-2$ for all $t\in [0,1]$.
}\end{example}

\begin{remark}
{\rm
Example \ref{nonsmoothexample}  shows 
that the weak necessary SMP can find not only the optimal control for minimization problem (any progressively measurable process taking values $-1$ or 1) but also the optimal control for maximization problem (the unique progressively measurable process taking value 0), which is in  the same spirit of the necessary condition for finite dimensional optimization. 
The Hamiltonian in Example \ref{nonsmoothexample} is nonsmooth in control variable $u$, which is beyond any known literature on SMP. 
}
\end{remark}

\begin{remark}
{\rm 
When $dx(t)=u(t)dW(t)$ and $U=[0,1]$ and everything else is kept the same as that in Example \ref{nonsmoothexample}, the problem is the same as that of  \cite[Example 3.3.1]{yong.zhou:stochasticcontrols}. 
Theorem \ref{WNSMP}  can again be applied to find the optimal control candidate $\bar u(t)=0$. (We leave this to the reader to check.)
The Hamiltonian  is a convex function of $u$ and $\bar u(t)=0$ is a minimum point. This is the reason that \cite{peng:SMP} introduces the generalized Hamiltonian ${\cal H}$ which makes $\bar u(t)=0$ a maximum point. 
}
\end{remark}

\subsection{Examples: Weak SMP with Regime-Switching}
\begin{example}(Quadratic Loss Minimization) 
{\rm
Here we adopt the setting in \cite[Section 6]{donnelly:SMPregimeswitching}. Let $ (\Omega,\mathcal{F},\left\{\mathcal{F}_t\right\}_{0\leq t\leq T},\mathbb{P}) $ be a complete probability space on which defined a 1-dimensional standard Brownian motion $ W $ and a continuous time Markov chain $ \alpha $ valued in a finite state space $ I=\{1,\cdots d\} $ with generator matrix $ Q=\left[ q_{ij} \right]_{i,j\in I} $ and initial mode $ \alpha(0)=i_0 $. Assume that $ W $ and $ \alpha $ are independent of each other and the filtration is generated jointly by $ W $ and $ \alpha $. Consider a market consisting of one risk-free bank account $ S_0=\left\{ S_0(t),t\in[0,T] \right\} $ and one risky stock $ S_1=\left\{ S_1(t), t\in[0,T] \right\} $. The risk-free asset's price process satisfies the following equation:
\begin{align*}
\left\{
\begin{array}{l}
dS_0(t)=r(t,\alpha(t-))S_0(t)dt \ t\in [0,T]\\
S_0(0)=1,
\end{array}
\right.
\end{align*}
where the risk-free rate of return $ r(t,i) $ is a bounded deterministic function for $ i\in I $. The price process of the risky stock is given by
\begin{align*}
\left\{
\begin{array}{l}
dS_1(t)=S_1(t)\left\{ b(t,\alpha(t-))dt+\sigma(t,\alpha(t-))dW(t) \right\} \ t\in[0,T]\\
S_1(0)=S_1>0,
\end{array}
\right.
\end{align*}
where the mean rate of return $ b(t,i) $ and the volatility $ \sigma(t,i) $ are bounded non-zero deterministic functions for $ i\in I $. Define the market price of risk $ \theta(t,i)\equiv\sigma^{-1}(t,i)(b(t,i)-r(t,i)) $.

Consider an agent with an initial wealth $ x_0>0 $. Let the $ \mathcal{F}_t $ previsible real valued process $ u(t) $ be the amount allocated to the stock at time $ t $. Then the wealth process $ x $ can be written as
\begin{equation}\label{qlmweatheqn}
\left\{
\begin{array}{l}
dx(t)=\left[ r(t,\alpha(t-))x(t)+u(t)\sigma(t,\alpha(t-))\theta(t,\alpha(t-)) \right]dt+u(t)\sigma(t,\alpha(t-))dW(t)\\
x(0)=x_0.
\end{array}
\right.
\end{equation}

A portfolio $ u(\cdot) $ is said to be admissible, written as $ u(\cdot)\in \mathcal{U}_{ad} $ if it is $ \mathcal{F}_t $-previsible, square integrable and such that the regime switching SDE \eqref{qlmweatheqn} has a unique solution $ x(\cdot) $ corresponding to $ u(\cdot) $. In this case, we refer to $ (x(\cdot),u(\cdot)) $ as an admissible pair. The agent's objective is to find an admissible pair $ (\bar{x}(\cdot),\bar{u}(\cdot)) $ such that 
\begin{equation*}
E\left( \bar{x}(T)-d \right)^2=\inf_{u\in\mathcal{U}_{ad}}E(x(T)-d)^2
\end{equation*}
for some fixed constant $ d\in \mathbb{R} $.

To solve this problem, first we find potential optimal candidate using Theorem \ref{WNSMP}. Suppose that $ (\bar{x}(\cdot),\bar{u}(\cdot)) $ is an optimal pair. Then the corresponding adjoint equation is 
\begin{equation}\label{qlmadjoint}
\left\{
\begin{array}{l}
dp(t)=-r(t,\alpha(t-))p(t)dt+q(t)dW(t)+s(t)\bullet dQ(t) \ t\in[0,T)\\
-p(T)=2\bar{x}(T)-2d.
\end{array}
\right.
\end{equation}

To find a solution $ (\bar{p},\bar{q},\bar{s}) $ to \eqref{qlmadjoint}, we try a process
\begin{equation}\label{qlmpbar}
\bar{p}(t)=\phi(t,\alpha(t))\bar{x}(t)+\psi(t,\alpha(t)),
\end{equation}
where $ \phi(t,i) $ and $ \psi(t,i) $ are deterministic smooth functions with terminal conditions 
\begin{align*}
\phi(T,i)=2 \text{ and } \psi(T,i)=-2d \text{ for } \forall i\in I.
\end{align*}
Applying Ito's formula to \eqref{qlmpbar} and comparing coefficients with \eqref{qlmadjoint} leads to
\begin{align}\label{qlmfirst}
&\begin{array}{c}
-r(t,\alpha(t-))\bar{p}(t)=\sum_{i=1}^d\mathcal{X}[\alpha(t-)=i]\bigg\{ \bar{x}(t)\left( \phi(t,i)r(t,i)+\Delta\phi(t,i)\right)\\ +\phi(t,i)\bar{u}(t)\sigma(t,i)\theta(t,i)+\Delta\psi(t,i)\bigg\},
\end{array}
\\\label{qlmsecond}
&\bar{q}(t)=\phi(t,\alpha(t-))\sigma(t,\alpha(t-))\bar{u}(t),\\
&\bar{s}_{ij}(t)=\bar{x}(t)(\phi(t,j)-\phi(t,i))+(\psi(t,j)-\psi(t,i)),
\end{align}
where for $ \varphi=\phi \text{ and }\psi $, denote by
\begin{equation*}
\Delta\varphi(t,i)\triangleq\varphi_t(t,i)+\sum_{j=1}^d q_{ij}(\varphi(t,j)-\varphi(t,i)).
\end{equation*}
The Hamiltonian is given by
\begin{equation}\label{qlmHamiltonian}
\begin{array}{r}
H(t,x,u,\alpha ,p,q)=r(t,\alpha)xp+u\sigma(t,\alpha)q+u\sigma(t,\alpha)\theta(t,\alpha)p.
\end{array}
\end{equation}
By Theorem \ref{WNSMP}, we have
\begin{equation*}
0\in\partial_u(-H)(t,\bar{x}(t),\bar{u}(t),\alpha(t-),\bar{p}(t),\bar{q}(t)).
\end{equation*}
Since $ H $ is a linear function of $ \bar{u} $, we must have
\begin{equation}\label{qlmqbar}
\bar{q}(t)=-\theta(t,\alpha(t-))\bar{p}(t).
\end{equation}
Substituting \eqref{qlmqbar} and \eqref{qlmpbar} into \eqref{qlmsecond} we obtain
\begin{equation}\label{qlmubar}
\bar{u}(t)=-\sigma^{-1}(t,\alpha(t-))\theta(t,\alpha(t-))(\bar{x}(t)+\phi^{-1}(t,\alpha(t-))\psi(t,\alpha(t-))).
\end{equation}
Substituting \eqref{qlmpbar} and \eqref{qlmubar} into \eqref{qlmfirst} leads to the following two differential equations
\begin{align}
\phi(t,i)(2r(t,i)-\vert \theta(t,i) \vert^2)+\Delta\phi(t,i)=0,\\
\psi(t,i)(r(t,i)-\vert \theta(t,i) \vert^2)+\Delta\psi(t,i)=0,
\end{align}
with terminal conditions
\begin{equation*}
\phi(T,i)=2 \text{ and } \psi(T,i)=-2d \text{ for } \forall i\in I.
\end{equation*}
It can be showed that the solutions are
\begin{align}\label{qlmphi}
\phi(t,i)=2E\bigg\{ \exp\bigg[ \int_t^T \left( 2r(s,\alpha(s))-\vert \theta(s,\alpha(s)) \vert^2\right) ds \bigg]\bigg| \alpha(t)=i \bigg\},\\\label{qlmpsi}
\psi(t,i)=-2dE\bigg\{ \exp\bigg[ \int_t^T \left( r(s,\alpha(s))-\vert \theta(s,\alpha(s)) \vert^2\right) ds \bigg]\bigg| \alpha(t)=i \bigg\}.
\end{align}
Detailed proofs can be found in \cite[Section 6]{donnelly:SMPregimeswitching} and \cite[Section 5]{zhangelliottsiu:SMPregimeswitchingjump}. Substituting \eqref{qlmphi} and \eqref{qlmpsi} back into \eqref{qlmubar} gives the potential optimal portfolio $ \bar{u} $ and the corresponding potential optimal wealth process $ \bar{x} $.

To verify the optimality of our candidate solution, we apply Theorem \ref{WSSMP}. Since \textbf{(A1)}-\textbf{(A4)} are satisfied, $ h(x(T),\alpha(T))\equiv (x(T)-d)^2 $ is convex and the Hamiltonian \eqref{qlmHamiltonian} is concave, we conclude that $ (\bar{x}(\cdot),\bar{u}(\cdot)) $ is indeed the optimal pair.
}
\end{example}
\begin{remark}
{\rm
Notice that in this case $ h $ is a convex function and the Hamiltonian is concave. Therefore, one can skip the necessary conditions and use a sufficient stochastic maximum principle of Pontryagin's type directly to find the optimal portfolio process. Detailed steps can be found in \cite[Section 6]{donnelly:SMPregimeswitching} and \cite[Section 5]{zhangelliottsiu:SMPregimeswitchingjump}. However, we follow a different approach here. Instead of using the sufficient SMP directly, we first find all admissible portfolios satisfying the necessary conditions stated in Theorem \ref{WNSMP}. Combining that with the adjoint equations, we then construct candidate optimal portfolio $ \bar{u} $. Finally, an application of Theorem \ref{WSSMP} confirms that $ \bar{u} $ is indeed the optimal portfolio. This approach is particularly useful when the conditions for sufficient SMP are not satisfied, e.g. nonconcave Hamiltonian.
}
\end{remark}

\begin{example}(Nonconcave Hamiltonian)
{\rm
Let $ (\Omega,\mathcal{F},\lbrace\mathcal{F}_t\rbrace_{0\leq t\leq 1},\mathbb{P}) $ be a complete probability space. Consider a one-dimensional Brownian motion $ W $ and a continuous time finite state Markov chain $ \lbrace\alpha(t)\vert t\in[0,1]\rbrace $ with state space $ I:=\lbrace 1,2 \rbrace $ and generator matrix $ Q:=[q_{ij}]_{i,j=1,2} $. Assume $ q_{12}+q_{21}\geq 2 $. Consider the following Markovian regime-switching control system
\begin{equation*}
\left\{ \begin{array}{ll}
dx(t)=u(t)dW(t), \ t\in[0,1]\\
x(0)=0
\end{array}
\right.
\end{equation*}
with the control domain $ U=[0,1] $ and the cost functional
\begin{align*}
J(u(\cdot))=E\bigg[ \int_0^1 \left( A(\alpha(t))u(t)+B(\alpha(t))u^2(t)+C(\alpha(t))x^2(t)\right) dt+D(\alpha(1))x^2(1) \bigg],
\end{align*}
where functions $ A,B,C,D: I\rightarrow\mathbb{R} $ satisfy
\begin{align*}
\left\{ \begin{array}{ll}
A(1)=-1\\
A(2)=0
\end{array},
\right.
\left\{ \begin{array}{ll}
B(1)=0\\
B(2)=-\frac{1}{2}
\end{array},
\right.
\left\{ \begin{array}{ll}
C(1)=0\\
C(2)=1
\end{array},
\right.
\left\{ \begin{array}{ll}
D(1)=\frac{1}{2}\\
D(2)=1
\end{array}.
\right.
\end{align*}

To solve this problem, first we find potential optimal solutions using Theorem \ref{WNSMP}. Suppose $ (\bar{x}(\cdot),\bar{u}(\cdot)) $ is an optimal pair. Then the corresponding adjoint equation is 
\begin{equation}\label{egadj}
\left\{ \begin{array}{ll}
d\bar{p}(t)=2C(\alpha(t))\bar{x}(t)dt+\bar{q}(t)dW(t)+\bar{s}(t)\bullet dQ(t)\\
\bar{p}(1)=-2D(\alpha(1))\bar{x}(1)
\end{array}.
\right.
\end{equation}
To find a solution $ (\bar{p},\bar{q},\bar{s}) $ to \eqref{egadj}, we try a process $ \bar{p}(t)=\phi(t,\alpha(t))\bar{x}(t) $, where $ \phi(t,i),\ i=1,2 $ are deterministic functions satisfying the terminal condition $ \phi(1,i)=-2D(i),\ i=1,2 $. Applying Ito's formula 
\begin{equation}\label{egito}
\begin{array}{c}
\displaystyle d\bar{p}(t)=\sum_{i=1}^2 \mathcal{X}[\alpha(t-)=i]\bigg\lbrace \bar{x}(t)\bigg( \phi_t(t,i)+\sum_{j=1}^2q_{ij}\left( \phi(t,j)-\phi(t,i) \bigg) \right) \bigg\rbrace dt \\
\displaystyle+\phi(t,\alpha(t))\bar{u}(t)dW(t)+\sum_{i\neq j}\bar{x}(t)\left( \phi(t,j)-\phi(t,i) \right)dQ_{ij}.
\end{array}
\end{equation}
Comparing the coefficients of \eqref{egadj} and \eqref{egito} leads to
\begin{align}\label{eglineareqn}
2C(\alpha(t))\bar{x}(t)=&\sum_{i=1}^2\mathcal{X}[\alpha(t-)=i]\bigg\lbrace \bar{x}(t)\bigg( \phi_t(t,i)+\sum_{j=1}^2q_{ij}(\phi(t,j)-\phi(t,i)) \bigg) \bigg\rbrace \\\label{egqcond}
\bar{q}(t)=&\phi(t,\alpha(t))\bar{u}(t)\\
\bar{s}_{ij}(t)=&\bar{x}(t)(\phi(t,j)-\phi(t,i))
\end{align}
As \eqref{eglineareqn} is a linear equation of $ \bar{x}(t) $, we guess that the coefficient of $ \bar{x}(t) $ vanishes at optimality and obtain the following two equations
\begin{align}\label{egode}
\left\{
\begin{array}{c}
\displaystyle-\phi_t(t,1)-q_{12}(\phi(t,2)-\phi(t,1))=0,\\
2-\phi_t(t,2)-q_{21}(\phi(t,1)-\phi(t,2))=0,
\end{array}
\right.
\end{align}
with terminal conditions
\begin{equation}\label{egterminalcond}
\phi(1,1)=-1 \text{ and } \phi(1,2)=-2.
\end{equation}
Solving the system of ordinary differential equations \eqref{egode} with terminal conditions \eqref{egterminalcond} gives
\begin{align*}
\left\{
\begin{array}{c}
\phi(t,1)=\dfrac{q_{12}(q_{12}+q_{21}-2)}{(q_{12}+q_{21})^2}\left( e^{(q_{12}+q_{21}-2)(t-1)}-1 \right)+\dfrac{2q_{12}}{q_{12}+q_{21}}(t-1)-1\\
\phi(t,2)=\dfrac{q_{21}(q_{12}+q_{21}-2)}{(q_{12}+q_{21})^2}\left( 1-e^{(q_{12}+q_{21}-2)(t-1)} \right)+\dfrac{2q_{12}}{q_{12}+q_{21}}(t-1)-2
\end{array}
\right.
\end{align*}
Moreover, since $ q_{12}+q_{21}\geq 2 $ and $ \frac{q_{21}(q_{12}+q_{21}-2)}{(q_{12}+q_{21})^2}<1 $, we obtain that $\phi(t,i)<-1, \ \forall t\in[0,1),i\in I $.
Consider the Hamiltonian 
\begin{align}\label{egHam}
\left\{
\begin{array}{l}
H(t,x,u,1,p,q)=u+qu\\
H(t,x,u,2,p,q)=\frac{1}{2}u^2-x^2+uq.
\end{array}
\right.
\end{align}
By Theorem \ref{WNSMP}, we have
\begin{align*}
0\in\partial_u(-H)(t,\bar{x}(t),\bar{u}(t),\alpha(t-),\bar{p}(t),\bar{q}(t))+N_U(\bar{u}(t)) \ \forall t\in[0,1], \ \mathbb{P}-\text{a.s.}
\end{align*}
Consequently on any nonzero measurable set $ E\in S=\Omega\times[0,1) $ such that $ \alpha(t-)=1 $, we can only have three cases:
\begin{description}
\item[Case 1]: $\bar{u}(t)=0\Rightarrow N_{[0,1]}(\bar{u}(t))=(-\infty,0] \text{ and } \bar{q}(t)+1\leq 0.$ \\According to \eqref{egqcond}, $ \phi(t,1)\bar{u}(t)\leq-1, \bar{u}(t)\geq-\frac{1}{\phi(t,1)}>0 $, contradiction.
\item[Case 2]: $ \bar{u}(t)=1\Rightarrow N_{[0,1]}(\bar{u}(t))=[0,+\infty) \text{ and } \bar{q}(t)+1\geq 0. $\\According to \eqref{egqcond}, $ \phi(t,1)\bar{u}(t)\geq-1, \bar{u}(t)\leq-\frac{1}{\phi(t,1)}<1 $, contradiction.
\item[Case 3]: $ 0<\bar{u}(t)<1 \Rightarrow N_{[0,1]}(\bar{u}(t))=\{0\} \text{ and } \bar{q}(t)=-1. $\\According to \eqref{egqcond}, $ \bar{u}(t)=-\frac{1}{\phi(t,1)}\in(0,1). $
\end{description}
Hence we conclude that $ \bar{u}(t)=-\frac{1}{\phi(t,1)} $ provided $ \alpha(t-)=1 $.
Similarly on any non-zero measurable set $ E\in S=\Omega\times[0,1) $ such that $ \alpha(t-)=2 $, we can only have three cases:
\begin{description}
\item[Case 1]: $\bar{u}(t)=1\Rightarrow N_{[0,1]}(\bar{u}(t))=[0,+\infty) \text{ and } \bar{q}(t)+\bar{u}(t)\geq 0.$ \\According to \eqref{egqcond}, $ (\phi(t,2)+1)\bar{u}(t)\geq 0, \bar{u}(t)\leq\frac{1}{\phi(t,2)+1}<0 $, contradiction.
\item[Case 2]: $ \bar{u}(t)\in(0,1)\Rightarrow N_{[0,1]}(\bar{u}(t))=\{0\} \text{ and } \bar{q}(t)+\bar{u}(t)=0. $\\According to \eqref{egqcond}, $ (\phi(t,2)+1)\bar{u}(t)=0, \bar{u}(t)=0 $, contradiction.
\item[Case 3]: $ \bar{u}(t)=0 \Rightarrow N_{[0,1]}(\bar{u}(t))=(-\infty,0] \text{ and } \bar{q}(t)+\bar{u}(t)\leq 0. $\\According to \eqref{egqcond}, $ (\phi(t,2)+1)\bar{u}(t)\leq 0, \bar{u}(t)=0. $
\end{description}
Hence we must have $ \bar{u}(t)=0 $ provided $ \alpha(t-)=2 $.\\
In conclusion, the potential optimal control can be written as
\begin{equation}\label{egoptimalcontrol}
\bar{u}(t)=-\frac{1}{\phi(t,1)}\mathcal{X}[\alpha(t-)=1].
\end{equation}

Let us now show that $ (\bar{x}(\cdot),\bar{u}(\cdot)) $ is indeed an optimal pair. Notice that the Hamiltonian \eqref{egHam} is not concave function of $ u $, and therefore Theorem \ref{WSSMP} cannot be applied. We have to use other methods to check the optimality of $ \bar{u} $. Given any admissible pair $ (x(\cdot),u(\cdot)) $, apply  Ito's formula on  $ \phi(t,\alpha(t))x^2(t) $ and write it in integral form,
\begin{equation}\label{egitointegral}
E\left[ \phi(1,\alpha(1))x^2(1) \right]=E\bigg[ \int_0^1 x^2(t)\bigg( \phi_t(t,\alpha(t))+\sum_{j=1}^2q_{ij}(\phi(t,j)-\phi(t,\alpha(t))) \bigg)+\phi(t,\alpha(t))u^2(t)dt \bigg].
\end{equation}
Substituting \eqref{egitointegral} into the cost functional and according to \eqref{eglineareqn},
\begin{align*}
J(u(\cdot))&=E\bigg[ \int_0^1 \left( A(\alpha(t))u(t)+B(\alpha(t))u^2(t)-\frac{1}{2}\phi(t,\alpha(t))u^2(t)\right) dt \bigg]\\
           &=E\bigg[ \int_{S_1}\left( -u(t)-\frac{1}{2}\phi(t,1)u^2(t)\right)dt+\int_{S_2}-\frac{1}{2}(1+\phi(t,2))u^2(t) dt \bigg]\\
           &=E\bigg[ \int_{S_1}\left(-\frac{1}{2}\phi(t,1)\left( u(t)+\frac{1}{\phi(t,1)}\right)^2+\frac{1}{2\phi(t,1)}\right) dt+\int_{S_2}-\frac{1}{2}(1+\phi(t,2))u^2(t)dt \bigg],
\end{align*}
where $ S_1\equiv\left\{ t\vert t\in[0,1] \text{ such that } \alpha(t-)=1 \right\} $ and $ S_2\equiv\left\{ t\vert t\in[0,1] \text{ such that } \alpha(t-)=2 \right\}=[0,1]\backslash S_1 $. Since $ \phi(t,1) \leq 1 $ and $ \phi(t,2)<1 \forall t\in[0,1] $, the minimum value of the cost functional is achieved at $ \bar{u} $ defined in \eqref{egoptimalcontrol}.
}
\end{example}

\section{Preliminary Results}
In this section, we introduce some preliminary results, which will be useful in the sequel. Hereafter, $ K $ represents a generic constant.
\subsection{Clarke's Generalized Gradient and Normal Cone}\label{ClarkeIntro}
In this subsection we recall some  basic concepts and  properties in nonsmooth analysis and optimization, which are needed in the statement and proof of the main results (Theorems \ref{WNSMP} and \ref{WSSMP}).    Clarke's generalized gradient is first introduced to the  finite dimensional space in \cite{clarke:gengradapp}  and then  extended to the infinite dimensional space in \cite{clarke:shadowprices, clarke:gengradfunctional}
and \cite{aubin:setvaluedanalysis}. Interested readers may refer to \cite{clarke:optimization}  for a detailed and complete treatment of the topic.

\begin{definition}(Generalized directional derivative)
Let $ C $ be an open subset of a Banach space $ X $, and let a function $ f:C\longrightarrow \mathbb{R} $ be given. We suppose that $ f $ is Lipschitz on $ C $. The generalized directional derivative of $ f $ at $ x $ in the direction $ v $, denoted $ f^o(x;v) $, is given by
\begin{equation*}
f^o(x;v)=\limsup_{y\rightarrow_C x, \lambda \downarrow 0} \dfrac{f(y+\lambda v)-f(y)}{\lambda}.
\end{equation*}
\end{definition}

\begin{definition}(Clarke's generalized gradient)
Let $ X^* $ denote the dual of $ X $ and $ \langle \cdot , \cdot \rangle $ be the duality pairing between $ X $ and $ X^* $. The generalized gradient of $ f $ at $ x $, denoted $ \partial{f}(x) $, is the set of all $ \zeta $ in $ X^* $ satisfying
\begin{equation*}
f^o(x;v)\geq \langle v, \zeta\rangle \textit{ for } \forall v\in X. 
\end{equation*}
\end{definition}

\begin{theorem}\label{minmaxgradient}
If $ f $ attains a local minimum or maximum at x, then $ 0\in \partial{f}(x) $.
\end{theorem}

Theorem \ref{minmaxgradient} is only valid in the case where $ C $ is open. When the function is defined on a general non-empty subset of $ X $, we need to introduce the so-called distance function and the concept of Clarke's tangent cone and normal cone.

\begin{definition}(Distance function)
Let $ X $ be a Banach space and $ C $ be a non-empty subset of $ X $. The distance function $ d_C:X\rightarrow \mathbb{R} $ is defined as
\begin{equation*}
d_C(x)=\inf\{\| x-c \|: c\in C\}.
\end{equation*}
\end{definition}

\begin{theorem}
The function $ d_C $ satisfies the following global Lipschitz condition on $ X $
\begin{equation*}
\vert d_C(x)-d_C(y) \vert \leq \| x-y \|.
\end{equation*}
\end{theorem}

\begin{definition}(Adjacent cone)\label{ajacentcone}
Let $\bar{C}$ be the closure of C and $ x\in \bar{C} $. The adjacent cone to C at x, denoted as $ T_C^b(x) $, is defined by
\begin{equation*}
T^b_C(x):=\{ v\vert \lim_{h\rightarrow 0^+} d_C(x+hv)/h = 0 \}.
\end{equation*}
\end{definition}

\begin{definition}(Tangent cone)
Suppose $ x\in C $. A vector $ v $ in $ X $ is a tangent to $ C $ at x provided $ d_C^o(x;v)=0 $. The tangent cone to C at x, denoted as $ T_C(x) $, is the set of all tangents to C at x.
\end{definition}

In addition, when the set C is convex, it can be proved that the adjacent and tangent cones coincide, see   \cite[Proposition 4.2.1]{aubin:setvaluedanalysis}.

\begin{theorem}\label{AdjacentTangent}
Assume that C is convex. Then $T_C(x) = T_C^b(x)$.
\end{theorem}

\begin{definition}(Normal cone)
Let $ x\in C $. The normal cone to C at x is defined by the polarity with $ T_C(x) $:
\begin{equation*}
N_C(x)=\{ \xi \in X^*: \langle \xi,v \rangle \leq 0 \textit{ for all } v\in T_C(x)\}.
\end{equation*}
\end{definition}

The following necessary optimality condition is proved in \cite[page 52 Corollary]{clarke:optimization}.

\begin{theorem} \label{NeceOptCond}
Assume that f is Lipschitz near x and attains a minimum over C at x. Then $ 0 \in \partial f(x)+N_C(x) $.
\end{theorem}

\subsection{Markovian Regime-Switching SDE and BSDE}
In this subsection, we establish the existence and uniqueness theorem of solutions to regime switching SDEs of the form \eqref{stateeqn}. First, we give the definition of the solution.
\begin{definition}\cite[Definition 3.11]{maoyuan:SDEswitching}
An $ \mathbb{R}^n $ valued stochastic process $ \{ x(t) \}_{0\leq t\leq T} $ is called a solution of equation \eqref{stateeqn} if it has the following properties:
\begin{enumerate}
\item $ \{ x(t) \} $ is continuous and $ \mathcal{F}_t $-adapted;
\item $ \{ b(t,x(t),u(t),\alpha(t-)) \}\in L^1_{\mathcal{F}}(S;\mathbb{R}^n) $ and $ \{ \sigma(t,x(t),u(t),\alpha(t-))\in L^2_{\mathcal{F}}(S;\mathbb{R}^{n\times m}) \} $;
\item for any $ t\in[0,T] $, equation 
\begin{align*}
x(t)=x_0+\int_0^t b(s,x(s),u(s),\alpha(s-))ds+\int_0^t \sigma(t,x(s),u(s),\alpha(s-))dW(s)
\end{align*}
holds with probability 1.
\end{enumerate}
A solution $ \{ x(t) \} $ is said to be unique if any other solution $ \{ \tilde{x}(t) \} $ is indistinguishable from $ \{ x(t) \} $, that is 
\begin{align*}
\mathbb{P}\{ x(t)=\tilde{x}(t) \mbox{ for all }0\leq t\leq T \}=1
\end{align*}
\end{definition}
Using the same method as in \cite[Chapter 3, Theorem 3.13]{maoyuan:SDEswitching}, the existence and uniqueness of solutions to regime-switching SDE of type \eqref{stateeqn} can be proved.
\begin{theorem}\label{exisuniqSDE}
Under assumption \textbf{(A1)}, given control $ u\in L^4(S;\mathbb{R}^k) $, there exists a unique solution $ x(t) $ to equation \eqref{stateeqn} and moreover, 
\begin{equation}\label{SDEEstimate}
E\left( \sup_{0\leq t\leq T}\vert x(t) \vert^2 \right)\leq K\left( 1+\vert x \vert^2 + \int_0^TE\vert u(t) \vert^2dt \right)
\end{equation}
for some constant $ K\geq 0 $.
\end{theorem}

We now develop results for existence and uniqueness of adapted solutions to regime switching BSDEs of type \eqref{ajointeqn}. Here we use the method of contraction mapping as in \cite[Chapter 6, Section 3]{yong.zhou:stochasticcontrols} and \cite[Chapter 6, Section 2]{pham:contimuoustimeSC} with the help of a martingale representation theorem for the joint filtration of a vector Brownian motion and a finite state Markov chain. 
Here we introduce the Dol\'{e}ans measure $ v_{[Q_{ij}]} $ on the measure space $ \left(\Omega\times[0,T],\mathcal{P}^\star\right) $:
\begin{align*}
v_{[Q_{ij}]}[A]:=E\int_0^T\mathcal{X}_A(\omega,t)d[Q_{ij}](t), \forall A\in\mathcal{P}^\star, \forall i,j\in I, i\neq j.
\end{align*}
By $ G=H \ v_{[Q]}$-a.e. for $ \mathbb{R}^{d\times d} $ mappings $ G $ and $ H $ on the set $ \Omega\times[0,T] $, we mean that
\begin{align*}
G_{ij}&=H_{ij} \ v_{[Q_{ij}]}\text{-a.e. } \forall i,j\in I, \ i\neq j\\
\text{and } G_{ii}&=H_{ii} \ \left( \mathbb{P}\otimes\text{Leb} \right)\text{-a.e. } \forall i\in I.
\end{align*}
We start by defining the following spaces for stochastic processes.
\begin{align*}
\mathbb{S}^2(\left[0,T\right]):=&\bigg\{Y:\Omega\times[0,T]\rightarrow\mathbb{R}^n\vert Y \text{ is } \mathcal{F}_t \text{ progressively measurable }\\ 
&\text{and } E\left(\sup_{0\leq t \leq T}\vert Y(t) \vert^2\right) < \infty\bigg\},\\
L^2\left(W,[0,T]\right)&:=\left\{\Lambda:\Omega\times[0,T]\rightarrow\mathbb{R}^{n\times m}\vert \Lambda\in\mathcal{P}^\star \text{ and } E\int_0^T\|\Lambda(t)\|^2 dt<\infty \right\},\\
L^2\left(Q,[0,T]\right)&:=\bigg\{ \Gamma=\left\{\left(\Gamma_{ij}^{(1)}\right)_{i,j=1}^d,\cdots ,\left(\Gamma_{ij}^{(n)}\right)_{i,j=1}^d\bigg\} \middle| \Gamma^{(l)}_{ii}=0 \ \mathbb{P}\otimes\text{Leb}-a.e. \forall i\in I, \right.\\
&\Gamma^{(l)}_{ij}\in \mathcal{P}^\star \text{ and } \sum_{l=1}^n\sum_{i,j=1}^dE\int_0^T\| \Gamma^{(l)}_{ij}(t) \|^2d\left[ Q_{ij} \right](t)<\infty \ \forall i,j\in I, i\neq j \bigg\}.
\end{align*}
It can be proved that $ L^2(W,[0,T]) $ and $ L^2(Q,[0,T]) $ are Hilbert spaces (see \cite[Lemma A.2.5]{DonnellyPhdThesis}). Next we present a martingale representation theorem for square integrable martingales with joint filtration generated by a Brownian motion and a finite state Markov chain. The proof can be found in \cite[Theorem B.4.6]{DonnellyPhdThesis} and \cite[Proposition 3.9]{donheu:regimeswitching}. 
\begin{theorem}\label{MRT}
Suppose the $ \mathbb{R}^n $-valued process $ \left\{Y(t),t\in[0,T]\right\} $ is a square-integrable $ \{\mathcal{F}_t\} $-martingale and null at the origin. Then there exists processes $ \Lambda\in L^2(W,[0,T]) $ and $ \Gamma\in L^2(Q,[0,T]) $ such that $ Y $ has the stochastic integral representation 
\begin{equation}\label{martingalerep}
Y(t)=Y(0)+\sum_{j=1}^m\int_0^t\Lambda_j(s)dW^j(s)+\int_0^t\Gamma(s)\bullet dQ(s) \ \text{a.s. }\forall t\in[0,T]  
\end{equation}
with the square-bracket quadratic variation process of $ Y $ given by 
\begin{align*}
\left[Y\right](t):=\sum_{i=1}^n\sum_{j=1}^m\int_0^t\Lambda_{ij}^2(s)ds+\sum_{l=1}^n\sum_{i,j=1}^d\int_0^t\left(\Gamma_{ij}^{\left(l\right)}(s)\right)^2d\left[Q_{ij}\right](t)\text{ a.s. }\forall t\in[0,T].
\end{align*}
Moreover, $ \Lambda $ and $ \Gamma $ are unique in the sense that if $ \tilde{\Lambda}\in L^2(W,[0,T]) $ and $ \tilde{\Gamma}\in L^2(Q,[0,T]) $ are such that \eqref{martingalerep} holds, then $ \Lambda=\tilde{\Lambda} \ \mathbb{P}\otimes\text{Leb}-a.e. $ and $ \Gamma=\tilde{\Gamma} \ v_{[Q]}-a.e.$
\end{theorem}

Suppose we are given a pair $ (\xi,f) $ called the terminal and generator satisfying the following conditions:
\begin{enumerate}
\item[(a)] $E\vert \xi \vert^2 <\infty $,
\item[(b)] $f: \Omega\times[0,T]\times\mathbb{R}^n\times\mathbb{R}^{n\times m}\rightarrow\mathbb{R}^n \text{ such that}$
\begin{enumerate}
\item[(i)] $ f(t,y,z) $ is $ \mathcal{F}_t $-progressively measurable for all $y,z $.
\item[(ii)] $ f(t,0,0)\in L_{\mathcal{F}}^2(S;\mathbb{R}^n) $,
\item[(iii)] $ f $ satisfies uniform Lipschitz condition in $ \left( y,z \right) $, i.e $\exists C_f>0 \text{ such that } $
\begin{align*}
\vert f(t,y_1,z_1)-f(t,y_2,z_2) \vert \leq C_f\left( \vert y_1-y_2 \vert+\vert z_1-z_2 \vert \right)
\end{align*}
$\forall y_1,y_2\in \mathbb{R}^n,z_1,z_2\in\mathbb{R}^{n\times m} \ \mathbb{P}\otimes\text{Leb} \ a.e. $
\end{enumerate}
\end{enumerate}
Consider the regime switching BSDE 
\begin{equation}\label{BSDE}
-dY(t)=f(t,Y(t),Z(t))dt-Z(t)dW(t)-S(t)\bullet dQ(t), \ Y(T)=\xi.
\end{equation}
\begin{definition}\label{RSBSDEdefn}
A solution to the regime switching BSDE \eqref{BSDE} is a set $ (Y,Z,S)\in\mathbb{S}^2([0,T])\times L^2(W,[0,T])\times L^2(Q,[0,T]) $ satisfying 
\begin{align*}
Y(t)=\xi+\int_t^T f(s,Y(s),Z(s)))ds-\int_t^T Z(s)dW(s)-\int_t^T S(s)\bullet dQ(t).
\end{align*}
\end{definition}
Now we prove the existence and uniqueness of a solution to the regime switching BSDE of type \eqref{BSDE}.
\begin{theorem}\label{RSBSDEtheorem}
Given a pair $ (\xi,f) $ satisfying $ (a) $ and $ (b) $, there exists a unique solution $ (Y,Z,S) $ to the regime switching BSDE \eqref{BSDE}.
\end{theorem}
The proof follows a contraction mapping argument similar to that in \cite[Chapter 6, Section 2]{pham:contimuoustimeSC}. For completeness, we give details in Appendix.

\subsection{A Moment Estimation}
In this subsection, we prove a moment estimation result. A simplified version of the moment estimate can be found in \cite[Chapter 3 Lemma 4.2 ]{yong.zhou:stochasticcontrols}.
\begin{lemma}\label{momest}
Let $ Y(t)\in L^2_{\mathcal{F}}(S;\mathbb{R}^n) $ be the solution of the following regime switching SDE
\begin{equation}
\left\{
\begin{array}{l}
dY(t)=[A(t)Y(t)+\beta(t)]dt+\displaystyle\sum_{j=1}^m\left[ B^j(t)Y(t)+\gamma^j(t) \right]dW^j(t)\\ 
Y(0)=y_0
\end{array}
\right.
\end{equation}
where $ A,B^j : \Omega\times[0,T]\rightarrow\mathbb{R}^{n\times n}$ and $ \beta,\gamma^j:\Omega\times[0,T]\rightarrow\mathbb{R}^n$ are $ \{\mathcal{F}_t\}_{t\geq 0} $-adapted and 
\begin{equation}
\left\{
\begin{array}{ll}
\vert A(t)\vert, \vert B^j(t)\vert \leq K \mbox{ a.e.} t\in[0,T], \mathbb{P}\mbox{-a.s.}\\
\displaystyle\int_0^TE\vert \beta(s) \vert^{2k}ds + \int_0^TE\vert \gamma^j(s) \vert^{2k}ds <\infty \mbox{ for some } k\geq 1.
\end{array}
\right.
\end{equation}
Then 
\begin{equation}\label{momesteqn}
\begin{array}{cc}
\sup\limits_{t\in[0,T]}E\vert Y(t) \vert^{2k} \leq K\left\{ E\vert y_0 \vert^{2k}+ \displaystyle\int_0^TE\vert \beta(s) \vert^{2k}ds+ \sum_{j=1}^m\int^T_0E\vert \gamma^j(s) \vert^{2k}ds \right\}
\end{array}
\end{equation}
\end{lemma}
\begin{proof}
For notation simplicity, we prove only the case $ m=n=1 $, leaving the case $ m,n>1 $ to the interested reader. We first assume that $ \beta,\gamma$ are bounded. Let $ \epsilon>0 $ and define 
\begin{equation}\label{approxabsoluteY}
\langle Y \rangle_\epsilon\triangleq\sqrt{\vert Y\vert^2+\epsilon^2}, \forall Y\in\mathbb{R}^L.
\end{equation}
Note that for any $ \epsilon>0 $, the map $ Y\rightarrow\langle Y \rangle_\epsilon $ is smooth and $ \langle Y \rangle_\epsilon\rightarrow\vert Y \vert $ as $ \epsilon\rightarrow 0 $. Applying Ito's formula to $ \langle Y(t) \rangle_\epsilon^{2k} $, we have
\begin{align*}
&d\langle Y(t) \rangle_\epsilon^{2k}=2k\langle Y(t) \rangle_\epsilon^{2k-1}\left.\dfrac{\vert Y(t) \vert}{\langle Y(t) \rangle_\epsilon}\right\{ \left[A(t)Y(t)+\beta(t)\right]dt+\left[B(t)Y(t)+\gamma(t)\right]dW(t)\bigg\} \\
&+\left[ k(2k-1)\langle Y(t) \rangle_\epsilon^{2k-2}\dfrac{\vert Y(t) \vert^2}{\langle Y(t) \rangle^2_\epsilon} +k\langle Y(t) \rangle_\epsilon^{2k-1}\dfrac{\epsilon^2}{\langle Y(t) \rangle^3_\epsilon} \right]\left[B(t)Y(t)+\gamma(t)\right]^2dt.
\end{align*}
Writing it in integral form and taking expectation. Since $ \langle Y(t) \rangle_\epsilon>\vert Y(t) \vert $ and $ 2k-1\geq1 $, we obtain
\begin{align*}
E\langle Y(t) \rangle^{2k}_\epsilon\leq & E\langle Y(0) \rangle^{2k}_{\epsilon}+2kE\int_0^t\langle Y(s) \rangle_{\epsilon}^{2k-1}\left\{ \vert A(s)\vert\langle Y(s) \rangle_\epsilon + \vert \beta(s) \vert \right\}ds\\
&+k(2k-1)E\int_0^t\langle Y(s) \rangle_\epsilon^{2k-2}\left[ \vert B(s)\vert\langle Y(s) \rangle_\epsilon +\vert\gamma(s)\vert \right]^2ds\\
\leq & E\langle Y(0) \rangle^{2k}_\epsilon+\left. KE\int_0^t\right\{\langle Y(s) \rangle_\epsilon^{2k}+\vert \beta(s) \vert\langle Y(s) \rangle^{2k-1}_\epsilon +\vert \gamma(s) \vert^2\langle Y(s) \rangle^{2k-2}_\epsilon\bigg\} ds,
\end{align*}
where $ K $ is a constant independent of $ t $. Applying \textit{Young's inequality}, we get
\begin{align*}
E\langle Y(t) \rangle_\epsilon^{2k}\leq & E\langle Y(0) \rangle^{2k}_\epsilon +KE\int_0^t\bigg\{\langle Y(s) \rangle_\epsilon^{2k}+\vert \beta(s) \vert^{2k}+\vert \gamma(s) \vert^{2k}\bigg\}ds.
\end{align*}
Finally, \textit{Gronwall's inequality} yields
\begin{equation}\label{momestlasteqn}
\begin{array}{ll}
\sup\limits_{t\in[0,T]}E\langle Y(t) \rangle^{2k}_\epsilon \leq & K\bigg\{ E\langle Y(0) \rangle^{2k}_\epsilon +E\displaystyle\int_0^T \left[ \vert \beta(s) \vert^{2k} + \vert \gamma(s) \vert^{2k}\right]ds\bigg\},
\end{array}
\end{equation}
for some constant $ K$. Letting $ \epsilon\rightarrow 0 $ in \eqref{approxabsoluteY}, then \eqref{momestlasteqn} becomes \eqref{momesteqn}.
\end{proof}

\subsection{Lipschitz Property}
\begin{lemma}\label{Liplemma1}
Let $ u_1, u_2 \in L^4(S;\mathbb{R}^k)$ and $ x_1, x_2 $ be the associated state processes satisfying \eqref{stateeqn}. The we have the following inequality:
\begin{align*}
\sup_{t\in[0,T]}E\vert x_1(t)-x_2(t) \vert^4 \leq K\|u_1-u_2\|^4
\end{align*}
\end{lemma}
\begin{proof}
Let $\xi(t) \triangleq x_1(t)-x_2(t)$. Then we have
\begin{align*}
d\xi(t)=&\left[ b(t,x_1(t),u_1(t),\alpha(t-))-b(t,x_2(t),u_2(t),\alpha(t-)) \right]dt\\
        &+\left[ \sigma(t,x_1(t),u_1(t),\alpha(t-))-\sigma(t,x_2(t),u_2(t),\alpha(t-)) \right]dW(t)
\end{align*}
For $ \varphi=b $ and $ \sigma $, let
\begin{equation}\label{Lip1}
\tilde{\varphi}_x(t)=\int_0^1\varphi_x\left(t,x_2(t)+\theta(x_1(t)-x_2(t)),u_1(t),\alpha(t-)\right)d\theta.
\end{equation}
Substitute \eqref{Lip1}, we obtain
\begin{align*}
d\xi(t)=&\left[ \tilde{b}_x(t)\xi(t)+b(t,x_2(t),u_1(t),\alpha(t-))-b(t,x_2(t),u_2(t),\alpha(t-)) \right]dt\\
       +&\left[ \tilde{\sigma}_x(t)\xi(t)+ \sigma(t,x_2(t),u_1(t),\alpha(t-))-\sigma(t,x_2(t),u_2(t),\alpha(t-))\right]dW(t).
\end{align*}
By Lemma \ref{momest}, we obtain
\begin{align*}
\sup_{t\in[0,T]}E\vert \xi(t) \vert^4 \leq & K\left\{ \int_0^TE\vert b(t,x_2(t),u_2(t),\alpha(t-))-b(t,x_2(t),u_1(t),\alpha(t-)) \vert^4 dt \right.\\
                                           & + \int_0^TE\vert \sigma(t,x_2(t),u_2(t),\alpha(t-))-\sigma(t,x_2(t),u_1(t),\alpha(t-)) \vert^4 dt\\
                                       \leq & K\left\{ \int_0^TE\vert u_1(t)-u_2(t) \vert^4dt \right\}
\end{align*}
\end{proof}
\begin{lemma}\label{lemmalocalLip}
The cost functional $ J:L^4(S;\mathbb{R}^k)\rightarrow\mathbb{R} $ is locally Lipschitz, i.e. for all $ \hat{u}\in L^4(S;\mathbb{R}^k) $, there exists a small ball $ B_{\hat{u}}^M $ with radius $ M>0 $ containing $ \hat{u} $ on which, we have 
\begin{equation}\label{localLip}
\vert J(u_1)-J(u_2) \vert \leq K_{M,\hat{u}}\| u_1-u_2 \|,
\end{equation}
for $ \forall u_1,u_2\in B_{\hat{u}}^M $, where $ K_{M,\hat{u}} $ is a constant dependent on $ MF\text{ and } \hat{u}$.
\end{lemma}
\begin{proof}
Given $ \hat{u}\in L^4(S;\mathbb{R}^k) $ and $ M>0 $, define 
\begin{align*}
B_{\hat{u}}^M \triangleq \left\{ u \in L^4(S;\mathbb{R}^k):\| u-\hat{u} \|<M \right\}.
\end{align*}
For any $ u_1,u_2 \in B^M_{\hat{u}} $ with associated state processes $ x_1,x_2 $, according to \textbf{(A4)}, we have 
\begin{align*}
&E\vert h(x_1(T),\alpha(T))-h(x_2(T),\alpha(T)) \vert\\
&\leq E\int_0^1\vert \langle h_x(x_1(T)+\theta(x_2(T)-x_1(T)),\alpha(T)),x_2(T)-x_1(T) \rangle \vert d\theta\\
&\leq K\left\{ E\left( 1+\vert x_1(T) \vert^2+\vert x_2(T) \vert^2 \right) \right\}^{\frac{1}{2}} \left\{ E\vert x_2(T)-x_1(T) \vert^2 \right\}^{\frac{1}{2}}\\
&\leq K\left\{ E(1+\vert \hat{x}(T) \vert^2+\vert \hat{x}(T)-x_1(T) \vert^2+\vert \hat{x}(T)-x_2(T) \vert^2 ) \right\}^{\frac{1}{2}} \left\{ E\vert x_2(T)-x_1(T) \vert^2 \right\}^{\frac{1}{2}}
\end{align*}
by H\"{o}lder's inequality and Minkowski's inequality. According to Theorem \ref{exisuniqSDE}, Jensen's inequality and Lemma \ref{Liplemma1},
\begin{align*}
E\vert h(x_1(T),\alpha(T))-h(x_2(T),\alpha(T)) \vert \leq &K_M\left\{ 1+\left( \int_0^TE\vert \hat{u}(t) \vert^2 dt \right)^{\frac{1}{2}} \right\}\left\{ E\vert x_2(T)-x_1(T) \vert^2 \right\}^{\frac{1}{2}}\\
\leq & K_M\left\{ 1+\left( \int_0^TE\vert \hat{u}(t) \vert^4 dt \right)^{\frac{1}{4}} \right\}\left\{ E\vert x_2(T)-x_1(T) \vert^4 \right\}^{\frac{1}{4}}\\
\leq & K_{M,\hat{u}}\| u_1-u_2 \|.
\end{align*}
On the other hand,
\begin{align*}
E&\int_0^T\vert f(t,x_1(t),u_1(t),\alpha(t)) - f(t,x_2(t),u_2(t),\alpha(t)) \vert dt\\
\leq &  E\int_0^T\left(\vert f(t,x_1(t),u_1(t),\alpha(t)) - f(t,x_2(t),u_1(t),\alpha(t)) \vert\right. \\
&+ \left.\vert f(t,x_2(t),u_1(t),\alpha(t)) - f(t,x_2(t),u_2(t),\alpha(t)) \vert\right) dt.
\end{align*}
Following similar arguments, we have
\begin{align*}
E\int_0^T\vert f(t,x_1(t),u_1(t),\alpha(t)) - f(t,x_2(t),u_1(t),\alpha(t)) \vert dt \leq K_{M,\hat{u}}\| u_1-u_2 \|.
\end{align*}
For the second term, by \textbf{(A3)} we have
\begin{align*}
&E\int_0^T\vert f(t,x_2(t),u_1(t),\alpha(t)) - f(t,x_2(t),u_2(t),\alpha(t)) \vert dt\\
&\leq E\int_0^T\left\{ K_1+K_2(\vert x_2(t) \vert+\vert u_1(t) \vert+\vert u_2(t) \vert)\right\}\vert u_1(t)-u_2(t) \vert dt\\
&\leq K\left\{ E\int_0^T( 1+\vert x_2(t) \vert^2+\vert u_1(t) \vert^2+\vert u_2(t) \vert^2 )dt \right\}^{\frac{1}{2}}\| u_1-u_2 \|\\
&\leq K_M\left( 1+\left\{ E\int_0^T\vert \hat{u}(t) \vert^4 dt \right\}^{\frac{1}{4}} \right)\| u_1-u_2 \|\\
&\leq K_{M,\hat{u}}\| u_1-u_2 \|,
\end{align*}
and \eqref{localLip} follows by combining the above inequalities.
\end{proof}

\subsection{Taylor Expansions}
Let $ (x,u) $ be an admissible pair. 
Let $ v\in\L^4(S;\mathbb{R}^k) $ and $ \epsilon > 0 $. Define $ u^\epsilon(t) \triangleq u(t)+\epsilon v(t) $ for all $ t\in[0,T] $. Let $ (x^\epsilon,u^\epsilon) $ satisfy the following stochastic control system:
\begin{equation*}
\left\{
\begin{array}{cl}    
    dx^\epsilon(t) =& b(t,x^\epsilon(t),u^\epsilon(t),\alpha(t-))dt + \sigma(t,x^\epsilon(t),u^\epsilon(t),\alpha(t-))dW(t), \ t\in[0,T],\\
    
    x^\epsilon(0) =& x_0\in\mathbb{R}^L, \alpha(0)=i_0\in I.
\end{array}\right.
\end{equation*}
Next, for $ \varphi=b,\ \sigma^{j}(1\leq j \leq M) \textit{ and } f $, we define
\begin{align*}
\left\{
\begin{array}{ll}    
 &\varphi_x(t)  \triangleq \varphi_x(t,x(t),u(t),\alpha(t-)), \\
 &\delta\varphi(t)  \triangleq \varphi(t,x(t),u^\epsilon(t),\alpha(t-))-\varphi(t,x(t),u(t),\alpha(t-)).
\end{array}\right.
\end{align*}
Let $ y^\epsilon $ be the solution of the following regime-switching SDE:
\begin{equation}\label{yepsilon}
\left\{
\begin{array}{cl} 
 dy^\epsilon(t)  =& \left\lbrace b_x(t)y^\epsilon(t)+\delta b(t)\right\rbrace dt+ \sum\limits_{j=1}^m \left\lbrace \sigma_x^j(t)y^\epsilon(t)+\delta\sigma^j(t) \right\rbrace dW^j(t), \
 t\in [0,T], \\

 y^\epsilon(0) =& 0, \alpha(0)=i_0\in I.
\end{array}\right.
\end{equation}

\begin{remark}
{\rm 
The variation in our proof is different from the so-called spike variation technique in the proof of Peng's maximum principle in \cite{peng:SMP} and \cite{yong.zhou:stochasticcontrols}. In their proof, where $ u^\epsilon(t)=u(t)+1_{[\tau,\tau+\epsilon]}v(t) $, one first perturbs an optimal control on a small set of size $ \epsilon $ and then let $ \epsilon \rightarrow 0 $. Whereas, in our proof we perturbs an optimal control over the whole space. Then reason behind this is that in the definition of Clarke's generalized directional derivative, $ v(t) $ represents a directional vector in $ L^4(S;\mathbb{R}^k) $ and must be fixed. One perturbs the control through multiplication of a scalar $ \epsilon $ and letting $ \epsilon \rightarrow 0 $.
}
\end{remark}

The following lemma gives the Taylor expansion result of the state process and cost functional.
\begin{lemma}\label{TaylorExp}
Let assumptions \textbf{(A1)}-\textbf{(A4)} hold. Then, we have
\begin{align}\label{firstTaylor}
\sup_{t\in[0,T]} E\left| x^\epsilon(t)-x(t) \right|^2 = O(\epsilon^2),\\\label{secondTaylor}
\sup_{t\in[0,T]} E\left| y^\epsilon(t) \right|^2 = O(\epsilon^2),\\\label{thirdTaylor}
\sup_{t\in[0,T]} E\left| x^\epsilon(t)-x(t)-y^\epsilon(t) \right|^2 = o(\epsilon^2).
\end{align}
Moreover, the following expansion holds for the cost functional:
\begin{equation}\label{fourthTaylor}
\begin{array}{cl}
J(u^\epsilon)= J(u)+E\langle h_x(x(T),\alpha(T)),y^\epsilon(t) \rangle + E\displaystyle\int_0^T \left\lbrace \langle f_x(t) ,y^\epsilon(t)\rangle + \delta f(t)\right\rbrace dt + o(\epsilon).
\end{array}
\end{equation}
\end{lemma}
\begin{proof}
For simplicity, we carry out the proof only for the case $ n=m=1 $.
\\

\noindent \textit{Proof of \eqref{firstTaylor}}.
Let $ \xi^\epsilon(t)\triangleq x^\epsilon(t)-x(t) $. The we have
\begin{equation}\label{xiepsilon}
\left\{
\begin{array}{cl}    
    d\xi^\epsilon(t) &= \left\{ \tilde{b}^\epsilon_x(t)\xi^\epsilon(t)+\delta b(t) \right\} dt + \left\{ \tilde{\sigma}_x^\epsilon(t)\xi^\epsilon(t)+\delta\sigma(t) \right\} dW(t)\\
    
    \xi(0) &= 0, \alpha(0)=i_0.
\end{array}\right.
\end{equation}
where for $\phi=b$ and $\sigma$, 
\begin{equation}\label{bsigmatilda}
\begin{array}{cl}    
    \tilde{\phi}^\epsilon_x(t) \triangleq \displaystyle\int^1_0 \phi_x(t,x(t)+\theta(x^\epsilon(t)-x(t)),u^\epsilon(t),\alpha(t-))d\theta.
\end{array}
\end{equation}
By Lemma \ref{momest}, since $ \tilde{b}^\epsilon_x(t) $, and $ \tilde{\sigma}^\epsilon_x(t) $ are bounded according to assumption \textbf{(A1)}, we obtain
\begin{align*}
\sup_{t\in[0,T]} E\vert \xi^\epsilon(t) \vert^2 & \leq K\int_0^T E \bigg\{ \vert \delta b(s) \vert^2 + \vert \delta \sigma(s) \vert^2 \bigg\} ds\\
& \leq K \epsilon^2 \int^T_0 E \vert v(s)\vert^2 ds \\
& \leq K \epsilon^2 .
\end{align*}
This proves \eqref{firstTaylor}.
\\
\noindent \textit{Proof of \eqref{secondTaylor}}.
Similarly, $ b_x(t) $ and $ \sigma_x(t) $ are bounded according to assumption \textbf{(A1)}. Applying Lemma \ref{momest} to \eqref{yepsilon}, we obtain
\begin{align*}
\sup_{t\in[0,T]} E\vert y^\epsilon(t) \vert^2 & \leq K\int_0^T E \bigg\{ \vert \delta b(s) \vert^2 + \vert \delta \sigma(s) \vert^2 \bigg\} ds\leq K \epsilon^2.
\end{align*}
This proves \eqref{secondTaylor}.
\\

\noindent \textit{Proof of \eqref{thirdTaylor}}.
Let $ \zeta^\epsilon(t)\triangleq x^\epsilon(t)-x(t)-y^\epsilon(t) \equiv \xi^\epsilon(t)-y^\epsilon(t) $. Then, by \eqref{xiepsilon} and \eqref{yepsilon} we have
\begin{align*}
d\zeta^\epsilon(t) = & d\xi^\epsilon(t)-dy^\epsilon(t)\\
                   = & \left\{ \tilde{b}_x^\epsilon(t)\xi^\epsilon(t)-b_x(t)y^\epsilon(t) \right\}dt+ \left\{ \tilde{\sigma}_x^\epsilon(t)\xi^\epsilon(t)-\sigma_x(t)y^\epsilon(t) \right\}dW(t)\\
                   = & \left\{ \tilde{b}_x^\epsilon(t)\zeta^\epsilon(t)+\left[ \tilde{b}_x^\epsilon(t)-b_x(t) \right]y^\epsilon(t) \right\}dt+ \left\{ \tilde{\sigma}_x^\epsilon(t)\zeta^\epsilon(t)+\left[ \tilde{\sigma}_x^\epsilon(t)-\sigma_x(t) \right]y^\epsilon(t) \right\}dW(t)
\end{align*}
Since $ \tilde{b}_x^\epsilon(t) $ and $ \tilde{\sigma}_x^\epsilon(t) $ are bounded by assumption \textbf{(A1)}, applying Lemma \ref{momest} we obtain
\begin{equation}\label{thirdTaylortosub}
\begin{array}{cl}
\sup_{t\in[0,T]} E\vert \zeta^\epsilon(t) \vert^2 \leq  K\displaystyle\int_0^T E\bigg\{ \left| \left[ \tilde{b}_x^\epsilon(t)-b_x(t) \right]y^\epsilon(t) \right|^2 + \left| \left[ \tilde{\sigma}_x^\epsilon(t)-\sigma_x(t) \right]y^\epsilon(t) \right|^2 \bigg\} dt.
\end{array}
\end{equation}
Recall that $ \bar{\omega} $ appearing in \textbf{(A4)} is a modulus of continuity. Thus for any $ \rho>0 $, there exists a constant $ K_{\rho}>0 $ such that 
\begin{equation}\label{modulusofcont}
\bar{\omega}(r)\leq \rho + rK_{\rho}, \ \forall r\geq 0.
\end{equation}
 By H\"older's inequality, \eqref{bsigmatilda}, \eqref{secondTaylor}, \eqref{firstTaylor} and \eqref{modulusofcont}, we have
\begin{align*}
&\int_0^T E\left| \left[ \tilde{b}^\epsilon_x(t)-b_x(t) \right]y^\epsilon(t) \right|^2 dt \\
& \leq \int_0^T \left( E\left| \tilde{b}_x^\epsilon(t)-b_x(t) \right|^4 \right)^{\frac{1}{2}}\left(E\left| y^\epsilon(t) \right|^4\right)^\frac{1}{2} dt\\
& \leq K\int_0^T \left\{E\int_0^1 \left| b_x\left(t,x(t)+\theta\xi^\epsilon(t),u^\epsilon(t),\alpha(t-)\right)-b_x(t)\right|^4 d\theta\right\}^{\frac{1}{2}}\epsilon^2 dt\\
& \leq K\int_0^T \left\{ E\left(\xi^\epsilon(t)^4+ \bar{\omega}(\epsilon v(t))^4\right) \right\}^{\frac{1}{2}}\epsilon^2 dt \\
&\leq K \int_0^T\left\{ \epsilon^4+E[\rho+K_\rho\epsilon |v(t)|]^4 \right\}^\frac{1}{2}dt\epsilon^2.
\end{align*}
Hence the first term in (\ref{thirdTaylortosub})  is $o(\epsilon^2)$. Similarly the second and third terms are also $o(\epsilon^2)$, which gives \eqref{thirdTaylor}.
\\
\noindent \textit{Proof of \eqref{fourthTaylor}}.
By definition of the cost functional \eqref{costfun}, we have
\begin{equation} \label{temp1}
\begin{array}{cl}
&J(u^\epsilon)-J(u)\nonumber \\ 
&= E\left\{ h(x^\epsilon(T),\alpha(T))-h(x(T),\alpha(T)) \right\} \\
&+ \displaystyle E\int_0^T \left\{f(t,x^\epsilon(t),u^\epsilon(t),\alpha(t))-f(t,x(t),u(t),\alpha(t))\right\}dt
\end{array}
\end{equation}
 For the first term on the right side of (\ref{temp1}) we have
\begin{align*}
&E\left\{ h(x^\epsilon(T),\alpha(T))-h(x(T),\alpha(T)) \right\} \\
&= E\int^1_0 \langle h_x(x(T)+\theta\xi^\epsilon(T),\alpha(T)),\xi^\epsilon(T) \rangle d\theta\\                                      
                                           &= E\langle h_x(x(T),\alpha(T)),y^\epsilon(T) \rangle + E\langle h_x(x(T),\alpha(T)),\zeta^\epsilon(T) \rangle\\
                                           &+ E\int^1_0 \langle h_x(x(T)+\theta\xi^\epsilon(T),\alpha(T))-h_x(x(T),\alpha(T)),\xi^\epsilon(T)\rangle d\theta.
\end{align*}
Then, by \eqref{firstTaylor}, \eqref{thirdTaylor}, \textbf{(A4)} and applying H\"{o}lder's inequality, we have
\begin{equation}\label{fourthTaylor1}
E\left\{ h(x^\epsilon(T),\alpha(T))-h(x(T),\alpha(T)) \right\} = E\langle h_x(x(T),\alpha(T)),y^\epsilon(T) \rangle + o(\epsilon).
\end{equation}
For the second term on the right side of (\ref{temp1}) we have
\begin{align*}
E& \int_0^T\left\{ f(t,x^\epsilon(t),u^\epsilon(t),\alpha(t))-f(t,x(t),u(t),\alpha(t)) \right\}dt\\
=& E\int^T_0 \left\{ \int_0^1 \langle f_x(t,x(t)+\theta\xi^\epsilon(t),u^\epsilon(t),\alpha(t)),\xi^\epsilon(t) \rangle d\theta \right\}\\
&+ \left\{ f(t,x(t),u^\epsilon(t),\alpha(t))-f(t,x(t),u(t),\alpha(t)) \right\} dt\\
=& E\int_0^T \left\{ \langle f_x(t),y^\epsilon(t) \rangle + \delta f(t)\right\}\\
&+ \left\{\int_0^1\langle f_x(t,x(t)+\theta\xi^\epsilon(t),u^\epsilon(t),\alpha(t))-f_x(t), y^\epsilon(t) \rangle d\theta\right\}\\
&+ \left\{\int_0^1\langle f_x(t,x(t)+\theta\xi^\epsilon(t),u^\epsilon(t),\alpha(t)),\zeta^\epsilon(t) \rangle d\theta \right\}dt
\end{align*}
Then, using \textbf{(A4)} and by a similar argument as in the proof of \eqref{thirdTaylor}, we have
\begin{equation}\label{fourthTaylor2}
\begin{array}{cl}
&E\displaystyle\int_0^T\left\{ f(t,x^\epsilon(t),u^\epsilon(t),\alpha(t))-f(t,x(t),u(t),\alpha(t)) \right\}dt \\
&=E\displaystyle\int_0^T \left\{ \langle f_x(t),y^\epsilon(t) \rangle + \delta f(t)\right\}+ o(\epsilon).
\end{array}
\end{equation}
\eqref{fourthTaylor} follows from \eqref{fourthTaylor1} and \eqref{fourthTaylor2}.
\end{proof}

\subsection{Duality Analysis}
\begin{lemma}
Let assumptions \textbf{(A1)}-\textbf{(A4)}  hold. Let $ y^\epsilon $ be the solution of \eqref{yepsilon} and $ (p,q,s) $ be the adapted solution of \eqref{ajointeqn}. Then
\begin{equation}\label{duality}
\begin{array}{cc}
E\langle p(T),y^\epsilon(T) \rangle = E\left. \displaystyle\int_0^T \right\{ \langle p(t),\delta b(t) \rangle + \langle f_x(t),y^\epsilon(t) \rangle+ tr\left( q(t)^\intercal\delta\sigma(t) \right)\bigg\} dt
\end{array}
\end{equation}
\end{lemma}

\begin{proof}
Applying Ito's lemma and taking expectation immediately lead to  \eqref{duality}.
\end{proof}
Now we are able to give the following lemma, which is of great importance.
\begin{lemma}\label{importlemma}
Let assumptions \textbf{(A1)}-\textbf{(A4)}  hold. For any $ \varepsilon > 0 $ and $ v \in L^4(S;\mathbb{R}^K) $, define 
\begin{equation*}
u^\epsilon(t) \triangleq u(t)+\epsilon v(t) \textit{ for } \forall t\in [0,T].
\end{equation*}
Then we have
\begin{equation*}
\begin{array}{cl}
J(u^\epsilon)-J(u) \\
= E\displaystyle\int_0^T (-H(t,x(t),u^\epsilon(t),\alpha(t-),p(t),q(t)))
-(-H(t,x(t),u(t),\alpha(t-),p(t),q(t))) dt +o(\epsilon)
\end{array}
\end{equation*}
\end{lemma}

\begin{proof}
According to Lemma \ref{TaylorExp}, we have
\begin{align*}
&J(u^\epsilon)-J(u) \\
&= E\langle h_x(x(T),\alpha(T)),y^\epsilon(T) \rangle 
                                 + E\displaystyle\int_0^T \left\{\langle f_x(t),y^\epsilon(t)\rangle + \delta f(t)\right\} dt + o(\epsilon)\\
                                 &= E\langle -p(T),y^\epsilon(T) \rangle 
                                 + E\displaystyle\int_0^T \left\{\langle f_x(t),y^\epsilon(t)\rangle + \delta f(t)\right\} dt + o(\epsilon).
\end{align*}
Applying \eqref{duality}, we obtain
\begin{align*}
&J(u^\epsilon)-J(u)= E\displaystyle\int_0^T -\bigg\{ \langle p(t),\delta b(t) \rangle + tr\left(q(t)^\intercal\delta\sigma(t)\right) -\delta f(t) \bigg\} dt+o(\epsilon)\\
&= E\displaystyle\int_0^T (-H(t,x(t),u^\epsilon(t),\alpha(t-),p(t),q(t)))-(-H(t,x(t),u(t),\alpha(t-),p(t),q(t))) dt +o(\epsilon)
\end{align*}
\end{proof}

\section{Proof of the Main Theorems}
\subsection{Proof of Theorem \ref{WNSMP} }
We follow the technique developed in \cite{clarke:shadowprices}.
Given an optimal 5-tuple $ (\bar{x},\bar{u},\bar{p},\bar{q},\bar{s}) $, define a functional $ \mathcal{H}^{\bar{u}}: L^4(S;\mathbb{R}^k)\rightarrow \mathbb{R} $ as following
\begin{equation*}
\mathcal{H}^{\bar{u}}(u) = E\int_0^T -H(t,\bar{x}(t),u(t),\alpha(t-),\bar{p}(t),\bar{q}(t))dt.
\end{equation*}
By a similar argument as in Lemma \ref{lemmalocalLip}, it can be proved that the functional $ \mathcal{H}^{\bar{u}} $ is also locally Lipschitz on $  L^4(S;\mathbb{R}^k) $. Next, we define  Clarke's generalized gradient of the functionals $ J $ and $ \mathcal{H}^{\bar{u}} $ at $ \bar{u} $ and explore their properties.
\begin{definition}\label{defJHtangent}
Let $L^{\frac{4}{3}}(S;\mathbb{R}^k)$ denote the dual space of $ L^4(S;\mathbb{R}^k) $ and $ \langle\cdot ,\cdot\rangle $ denote the duality pairing between $L^4(S;\mathbb{R}^k) $ and $L^{\frac{4}{3}}(S;\mathbb{R}^k)$. Given an admissible control $ \bar{u}\in L^4(S;\mathbb{R}^k) $,  Clarke's generalized gradient of $ J $ at $ \bar{u} $, denoted by $ \partial J(\bar{u}) $, is the set of all $ \zeta\in L^{\frac{4}{3}}(S;\mathbb{R}^k) $ satisfying 
\begin{equation}\label{ClarkeJ}
J^o(\bar{u};v)=\limsup_{u\rightarrow\bar{u},\epsilon\rightarrow 0}\dfrac{J(u+\epsilon v)-J(u)}{\epsilon} \geq \langle v,\zeta \rangle,
\end{equation}
for all $v\in L^4(S;\mathbb{R}^k)$.
  Clarke's generalized gradient of $ \mathcal{H}^{\bar{u}} $ at $ \bar{u} $ is defined similarly. 
\end{definition}

Then, according to Lemma \ref{importlemma}, given $ u\in L^4(S;\mathbb{R}^k) $, for any $ \epsilon >0 $ and $ v\in L^4(S;\mathbb{R}^k) $ such that $ u+\epsilon v\in L^4(S;\mathbb{R}^k) $, we have
\begin{equation*}
J(u+\epsilon v)-J(u) = \mathcal{H}^{\bar{u}}(u+\epsilon v) - \mathcal{H}^{\bar{u}}(u) +o(\epsilon).
\end{equation*}
Hence, we have 
\begin{equation*}
J^o(\bar{u};v) = (\mathcal{H}^{\bar{u}})^o(\bar{u};v), \ \textit{for } \forall v\in L^4(S;\mathbb{R}^k).
\end{equation*}
Therefore, by Definition \ref{defJHtangent}, we conclude
\begin{equation*}
\partial J(\bar{u}) = \partial \mathcal{H}^{\bar{u}}(\bar{u}).
\end{equation*}
Since $ \bar{u} $ is an optimal control on $ \mathcal{U}_{ad} $, according to Theorem \ref{NeceOptCond},
\begin{equation}\label{optcondnormalcone}
0 \in \partial J(\bar{u})+N_{\mathcal{U}_{ad}}(\bar{u}) = \partial\mathcal{H}^{\bar{u}}(\bar{u})+N_{\mathcal{U}_{ad}}(\bar{u}).
\end{equation}

To characterize Clarke's tangent cone in the $ L^4(S;\mathbb{R}^k) $ space, we recall \cite[Theorem 8.5.1]{aubin:setvaluedanalysis}.
Let $ (\Omega,S,\mu) $ be a complete $ \sigma $-finite measure space and X be a separable Banach space. Consider a measurable set-valued map $K: \Omega\leadsto X $.
We associate with it the subset $ \mathcal{K} \subset L^p(\Omega;X,\mu) $ of selections defined by
\begin{align*}
\mathcal{K}:=\{ x\in L^p(\Omega;X,\mu) \vert \textit{ for almost all }\omega\in\Omega, x(\omega)\in K(\omega) \}.
\end{align*}
\begin{theorem}\label{TangLeb}
Assume that the set-valued map K is measurable and has closed images. Then for every $ x\in\mathcal{K} $, the set valued map
$\omega \rightarrow T_{K(\omega)}^b(x(\omega))$
is measurable. Furthermore
\begin{equation*}
\{ v\in L^p(\Omega;X,\mu)\vert \textit{ for almost all  } \omega, v(\omega)\in T_{K(\omega)}^b(x(\omega)) \} \subset T_{\mathcal{K}}^b(x).
\end{equation*}
\end{theorem}
Returning to our proof, since $ U $ is convex, by definition, $ \mathcal{U}_{ad} $ is also a convex subset of $ L^4(S;\mathbb{R}^k) $. Therefore, by Theorem \ref{AdjacentTangent} and Theorem \ref{TangLeb}, we obtain
\begin{equation}\label{AccTangLeb}
T_{\mathcal{U}_{ad}}(\bar{u}) \supset \{ v\in L^4(S;\mathbb{R}^k) \vert v(\omega,t)\in T_{U}(\bar{u}(\omega,t)) \ \mu\textit{-almost surely}\}.
\end{equation}
The optimality  condition \eqref{optcondnormalcone} together with \eqref{AccTangLeb} implies that $ \exists \zeta \in L^{\frac{4}{3}}(S;\mathbb{R}^k) $ such that
\begin{equation}\label{twoinequalities}
\left\{
\begin{array}{cl}
&E\displaystyle\int_0^T\langle \zeta(t),v(t) \rangle dt \leq 0 \textit{ for } \forall v\in L^4(S;\mathbb{R}^k)\textit{ such that } \\
&v(t)\in T_{U}(\bar{u}(t)) \ \textit{ for every } t\in[0,T], \mathbb{P}\textit{-almost surely} \\
&(\mathcal{H}^{\bar{u}})^o(\bar{u};v) + E\displaystyle\int_0^T \langle\zeta(t),v(t) \rangle dt \geq 0 \textit{ for } \forall v \in L^4(S;\mathbb{R}^k).
\end{array}
\right.
\end{equation}

Now, we recall a version of the measurable selection theorem in \cite{aubin:setvaluedanalysis}. 
\begin{definition}\cite[Definition 8.1.2]{aubin:setvaluedanalysis}
Let $ (\Omega,\mathcal{A}) $ be a measurable space and $ X $ be a complete separable metric space. Consider a set-valued map $ F: \Omega\leadsto X $. A measurable map $ f:\Omega\mapsto X $ satisfying 
\begin{align*}
\forall \omega \in \Omega, f(\omega)\in F(\omega)
\end{align*}
is called a measurable selection of $ F $.
\end{definition}
\begin{theorem}\cite[Theorem 8.1.3]{aubin:setvaluedanalysis}\label{measelectionthm}
Let $ X $ be a complete separable metric space, $ (\Omega, \mathcal{A}) $ a measurable space, $ F $ a measurable set-valued map from $ \Omega $ to closed nonempty subsets of $ X $. Then there exists a measurable selection of $ F $.
\end{theorem}

Return to our problem. Fix $ u\in\mathcal{U}_{ad} $ and $ (\omega,t)\in S $. Let $ \mathbb{Q}_+ $ denote the set of all strictly positive rationals. Following the argument in \cite[Page 325]{aubin:setvaluedanalysis} , we have
\begin{align*}
T_U(u(\omega,t)) &= T_U^b(u(\omega,t))
                                  =\bigcap_{n>0} cl\left( \bigcup_{\alpha\in\mathbb{Q}_+} \bigcap_{h\in[0,\alpha]\cap \mathbb{Q}_+}\dfrac{U-u(\omega,t)}{h}+\frac{1}{n}B \right),
\end{align*}
where $ B $ denotes the unit ball centred at 0. By \cite[Theorem 8.2.4]{aubin:setvaluedanalysis}, we conclude that the set-valued function $ T_{U}(\bar{u}) $ is measurable.

For the first inequality in \eqref{twoinequalities}, let $ M > 0 $ and define $ \bar{B}_M \triangleq \{ v\in \mathbb{R}^k : \|v\|\leq M \} $. For any positive integer $ n $, define a set-valued function $ \Pi_n^M $ as follows
\begin{equation*}
\Pi_n^M (\omega,t) = \left\{ \begin{array}{cl}
          &\{0\}, \ \textit{ if  } \langle \zeta(\omega,t),v \rangle < \dfrac{1}{n}, \ \forall v\in \bar{B}_M \cap T_U(\bar{u}(\omega,t)) \\        
          &\{ v\in \bar{B}_M\cap T_U(\bar{u}(\omega,t)):  \langle \zeta(\omega,t),v \rangle \geq \dfrac{1}{n} \}, \textit{  otherwise}.
        \end{array} \right.
\end{equation*}

The map $ (\omega,t,v)\rightarrow \langle \zeta(\omega,t),v \rangle $ is continuous in $ v $. Moreover, since $ \mathbb{R}^k $ is separable, the map can be expressed as the upper limit of a countable family of measurable functions and therefore is measurable. Therefore $ \Pi_n^M $ is measurable since countable intersection of measurable set-valued functions is still measurable. Hence, by Theorem \ref{measelectionthm}, $ \Pi_n^M $ admits a measurable selection $ v_n^M \in L^4(S;\mathbb{R}^k) $. Note that \eqref{twoinequalities} implies that the set
\begin{equation*}
\{ (\omega,t):\Pi_n^M(\omega,t)\neq \{0\} \}
\end{equation*}
must have $ \mu $ measure 0. Hence, we conclude that there exists a set, denoted as $ S_n^M $, where
\begin{equation*}
S_n^M = \{(\omega,t): \Pi^M_n(\omega,t)=\{0\}\}
\end{equation*}
and $ \mu(S_n^M)=1 $. Consequently, we have
\begin{equation}\label{measelecineqfirstineq}
\langle \zeta(\omega,t),v \rangle < \dfrac{1}{n} \ \ \forall v\in \bar{B}_M\cap T_{U}(\bar{u}(\omega,t)) \textit{  on } S^M_n.
\end{equation}
Define $ S^M=\bigcap_{n=1}^\infty S^M_n $ with $ \mu(S^M)=1 $ since $ \mu(S_n^M)=1 \ \forall n\in\mathbb{N}$. Moreover, since \eqref{measelecineqfirstineq} holds for all $ n $, we have
\begin{equation}\label{measelecineq2firstineq}
\langle \zeta(\omega,t),v \rangle \leq 0 \ \forall v\in\bar{B}_M\cap T_{U}(\bar{u}(\omega,t)) \textit{  on } S^M.
\end{equation}
Since \eqref{measelecineq2firstineq} holds for arbitrary $ M $, we obtain that
\begin{equation}\label{necetocomb1}
\langle \zeta(\omega,t),v \rangle \leq 0 \textit{ for } \forall v\in T_{U}(\bar{u}(\omega,t)) \  \textit{$\mu$-almost surely}.
\end{equation}

Next, we consider the second inequality in \eqref{twoinequalities}.
Define the partial generalized directional derivative of the Hamiltonian $ H $ at $ \bar{u}(t) $ in the direction $ v(t) $ as
\begin{equation*}
\begin{array}{cl}
-H_u^o(\bar{u}(t);v(t))= \limsup\limits_{u\rightarrow\bar{u},\epsilon\rightarrow 0}\dfrac{1}{\epsilon}\bigg\{-H(t,\bar{x}(t), u(t)+\epsilon v(t),\alpha(t-),\bar{p}(t),\bar{q}(t),s(t))\\
+H(t,\bar{x}(t), u(t),\alpha(t-),\bar{p}(t),\bar{q}(t),s(t))\bigg\}.
\end{array}
\end{equation*}
Using Fatou's Lemma on the second inequality in \eqref{twoinequalities}, we have
\begin{align}
& E\left\{ \int_0^T -H_u^o(\bar{u}(t);v(t))+\langle \zeta(t),v(t) \rangle dt\right\} 
 \geq (\mathcal{H}^{\bar{u}})^o(\bar{u};v)+E\int_0^T\langle \zeta(t),v(t) \rangle dt \geq 0\label{afterfatou}
\end{align}
\nocite{*}

Let $ M > 0 $ and define $ \bar{B}_M \triangleq \{ v\in \mathbb{R}^k : \|v\|\leq M \} $. For any $ n \in \mathbb{N}$, define a set-valued function $ \Gamma_n^M $ as follows
\begin{equation*}
\Gamma_n^M (\omega,t) = \left\{ \begin{array}{cl}
          &\{0\}, \ \textit{ if  } -H_u^o(\bar{u}(\omega,t);v)+\langle \zeta(\omega,t),v \rangle > -\dfrac{1}{n} \ \forall v\in \bar{B}_M \\     
          &\{ v\in \bar{B}_M: -H_u^o(\bar{u}(\omega,t);v) +\langle \zeta(\omega,t),v \rangle \leq -\dfrac{1}{n} \}, \textit{  otherwise}.
        \end{array} \right.
\end{equation*}
Using a similar argument as above, with the help of Theorem \ref{measelectionthm} and 
\eqref{afterfatou}, we can show that  the set
$\{ (\omega,t):\Gamma_n^M(\omega,t)\neq \{0\} \}$
must have $ \mu $ measure 0, which implies that
\begin{equation}\label{necetocomb2}
-H_u^o(\bar{u}(\omega,t);v))+\langle \zeta(\omega,t),v \rangle \geq 0 \  \textit{$\mu$-almost surely}.
\end{equation}
Combining \eqref{necetocomb1} and \eqref{necetocomb2}, we conclude
\begin{equation*}
0\in \partial_u(-H)(t,\bar{x}(t),\bar{u}(t),\alpha(t-),\bar{p}(t),\bar{q}(t)) + N_U(\bar{u}(t)),\ a.e. t\in[0,T], \ 
\textit{$\mathbb{P}$-a.s}.
\end{equation*}

\subsection{Proof of Theorem \ref{WSSMP} }
Given admissible pair $ (x,u) $, define
\begin{equation*}
H(t,x(t),u(t)) \triangleq H(t,x(t),u(t),\alpha(t-),\bar{p}(t),\bar{q}(t))  \textit{ for }  \forall t\in[0,T],\ \mathbb{P}\textit{-a.s.}
\end{equation*}
Under the convexity condition,  Clarke's generalized gradient and normal cone coincide with the subdifferential and normal cone  in the sense of convex analysis. Moreover, combining \eqref{neceSMPcond} and the concavity of $ H(t,\bar{x}(t),\cdot) $ for all $ t\in [0,T] $ a.s, we conclude that 
\begin{equation*}
H(t,\bar{x}(t),\bar{u}(t)) = \max_{u\in U}H(t,\bar{x}(t),u),\ \textit{a.e. } t\in[0,T],\ \mathbb{P}\textit{-a.s.}
\end{equation*}
Define $ \xi(t)\triangleq x(t)-\bar{x}(t) $ satisfying 
\begin{equation*}
\left\{
\begin{array}{cl}     
    d\xi(t) &= \left\{b(t,x(t),u(t),\alpha(t-))-b(t,\bar{x}(t),\bar{u}(t),\alpha(t-))\right\}dt \\
    &+ \sum\limits_{j=1}^m\left\{\sigma^j(t,x(t),u(t),\alpha(t-))-\sigma^j(t,\bar{x}(t),\bar{u}(t),\alpha(t-))\right\}dW^j(t),\ t\in[0,T],\\
    \xi(0) &= 0, \alpha(0)=i_0.
\end{array}\right.
\end{equation*}
Following a standard separating hyperplane argument in convex analysis (see \cite[Chapter 5]{rockafeller:convexanalysis}), we obtain
\begin{equation}\label{sufficientclaim}
\int_0^T \left\{ H(t,x(t),u(t))-H(t,\bar{x}(t),\bar{u}(t)) \right\} \leq \int_0^T\langle H_x(t,\bar{x}(t),\bar{u}(t)),\xi(t) \rangle dt
\end{equation}
for any admissible pair $ (x,u) $. Detailed proof of \eqref{sufficientclaim}  can be found in \cite{oksendal:SSMP}.

Applying Ito's formula to $\langle \bar{p}(t),\xi(t) \rangle$, noting the convexity of $ h $, 
the inequality (\ref{sufficientclaim})   and the definition of the Hamilitonian \eqref{Hamiltonian}, we have
\begin{align*}
&E\{h(x(T),\alpha(T))-h(\bar{x}(T),\alpha(T))\}\\
\geq &E\langle h_x(\bar x(T),\alpha(T)),\xi(T\rangle)\rangle\\
 =& -E\langle \bar{p}(T),\xi(T) \rangle \\
= & E \int_0^T \bigg\{ \langle H_x(t,\bar{x}(t),\bar{u}(t)),\xi(t) \rangle \\
&- \langle \bar{p}(t),b(t,x(t),u(t),\alpha(t-))-b(t,\bar{x}(t),\bar{u}(t),\alpha(t-)) \rangle\\
&-\sum_{j=1}^m\langle \bar{q}^j(t),\sigma^j(t,x(t),u(t),\alpha(t-))-\sigma^j(t,\bar{x}(t),\bar{u}(t),\alpha(t-))\rangle \bigg\} dt \\
\geq & -E\int_0^T\{ f(t,x(t),u(t),\alpha(t-))-f(t,\bar{x}(t),\bar{u}(t),\alpha(t-)) \}dt.
\end{align*}
Therefore $J(\bar{u}) \leq J(u)$
for all $ u\in \mathcal{U}_{ad} $.

\section{Conclusion}
We have proved in the paper a weak version of the necessary and sufficient stochastic maximum principle in a regime-switching diffusion model. Instead of insisting on the maximum condition of the Hamiltonian, we showed that $ 0 $ belongs to the sum of Clarke's generalized gradient of $ -H $ and Clarke's normal cone at the optimal control $ \bar{u} $, which also removes the requirement of the differentiability of the functions in the control variable. Under certain concavity conditions on the Hamiltonian, the necessary condition becomes sufficient. The theorem does not involve any second order terms, hence the second order differentiability of the functions in the state variable  is not required. Moreover, the absence of the second order adjoint equation considerably simplifies the SMP. Futher research on this topic includes the extension  of the weak SMP to more general stochastic control systems such as nonconvex control constraints and locally Lipschitz coefficients. We are currently working on these problems.

\bigskip
{\noindent\bf Acknowledgment}. The authors are grateful to Professor Nicole El Karoui for the useful discussions on the paper, especially on the contents of the measurability of stochastic processes.

\appendix 
\section{Appendix}
\subsection{Proof of Theorem \ref{RSBSDEtheorem}}

\begin{proof}
Consider the function $ \Phi $ on $ \mathbb{S}^2([0,T])\times L^2(W,[0,T])\times L^2(Q,[0,T]) $ mapping $ (Y,Z,S)\in \mathbb{S}^2([0,T])\times L^2(W,[0,T])\times L^2(Q,[0,T]) $ to $ \left(\hat{Y},\hat{Z},\hat{S}\right)=\Phi(Y,Z,S) $ defined by
\begin{equation*}
\hat{Y}(t)=\xi+\int_t^Tf(s,Y(s),Z(s))ds-\int_t^TZ(s)dW(s)-\int_t^TS(s)\bullet dQ(s).
\end{equation*}
Consider the square-integrable martingale 
\begin{equation*}
M(t)=E\left[\xi+\int_0^Tf(s,Y(s),Z(s))ds\middle|\mathcal{F}_t\right].
\end{equation*}
According to Theorem \ref{MRT}, there exists unique $ \left( \hat{Z},\hat{S} \right)\in L^2(W,[0,T])\times L^2(Q,[0,T]) $ such that
\begin{equation*}
M(t)=M(0)+\int_0^t\hat{Z}(s)dW(s)+\int_0^t\hat{S}(s)\bullet dQ(s).
\end{equation*}
We then define the process $ \hat{Y}(t) $ by 
\begin{align*}
\hat{Y}(t)&=E\left[ \xi+\int_t^Tf(s,Y(s),Z(s))ds\middle|\mathcal{F}_t \right]\\
          &=M(t)-\int_0^tf(s,\alpha(s),Y(s),Z(s))ds\\
          &=M(0)+\int_0^t\hat{Z}(s)dW(s)+\int_0^t\hat{S}(s)\bullet dQ(s)-\int_0^tf(s,Y(s),Z(s))ds\\
          &=\xi+\int_t^Tf(s,Y(s),Z(s))ds-\int_t^T\hat{Z}(s)dW(s)-\int_t^T\hat{S}(s)\bullet dQ(s).
\end{align*}
By Doob's $ L^2 $ inequality, we have
\begin{align*}
&E\left[\sup_{0\leq t\leq T}\middle|\int_t^T\hat{Z}(s)dW(s)\middle|\right]\leq 4E\left[ \int_0^T|\hat{Z}(s)|^2ds \right]<\infty,\\
&E\left[\sup_{0\leq t\leq T}\middle|\int_t^T\hat{S}(s)\bullet dQ(s)\middle|\right]\leq 4E\left[ \sum_{l=1}^n\sum_{i,j=1}^d\int_0^T|\hat{S}_{ij}^{(l)}(s)|^2d[Q_{ij}](s) \right]<\infty.
\end{align*}
Under the assumptions on $ (\xi,f) $, we conclude that $ \hat{Y}\in S^2([0,T]) $. Hence $ \Phi $ is a well defined function from $ S^2([0,T])\times L^2(W,[0,T])\times L^2(Q,[0,T]) $ into itself. Next, we show that $ (\hat{Y},\hat{Z},\hat{S}) $ is a solution to the regime switching BSDE \eqref{BSDE} if and only if it is a fixed point of $ \Phi $.

Let $ (U,V,\Gamma) $, $ (U',V',\Gamma') \in S^2([0,T])\times L^2(W,[0,T])\times L^2(Q,[0,T])$. Apply function $ \Phi $ and obtain $ (Y,Z,S)=\Phi(U,V,\Gamma),\ (Y',Z',S')=\Phi(U',V',\Gamma') $. Set $ (\bar{U},\bar{V},\bar{\Gamma}) =(U-U',V-V',\Gamma-\Gamma')$, $ (\bar{Y},\bar{Z},\bar{S})=(Y-Y',Z-Z',S-S') $ and $ \bar{f}(t)=f(t,U(t),V(t))-f(t,U'(t),V'(t)) $. Take $ \beta>0 $ to be chosen later and apply Ito's formula to $ e^{\beta s}\vert \bar{Y} \vert^2 $ on $ [0,T] $,
\begin{equation}\label{BSDEito}
\begin{array}{cl}
\vert \bar{Y}(0) \vert^2 =& -\displaystyle\int_0^Te^{\beta t}\left( \beta\vert\bar{Y}(t)\vert^2-2\bar{Y}(t)^\intercal\bar{f}(t) \right)dt-\int_0^Te^{\beta t}\vert \bar{Z}(t) \vert^2dt\\
&-\displaystyle\int_0^Te^{\beta t}\sum_{l=1}^n\sum_{i,j=1}^d\vert\bar{S}^{(l)}_{ij} \vert^2d\left[Q_{ij}\right](t)-2\int_0^T e^{\beta t}\bar{Y}(t)^\intercal\bar{Z}(t)dW(t)\\
&-2\displaystyle\int_0^Te^{\beta t}\sum_{l=1}^n\sum_{i,j=1}^d\bar{Y}^{(l)}(t)\bar{S}_{ij}^{(l)}(t)dQ_{ij}(t).
\end{array}
\end{equation}

Observe that, according to Young's inequality
\begin{align*}
&E\left[\left( \int_0^T e^{2\beta t}\vert \bar{Y}(t) \vert^2\vert \bar{Z}(t) \vert^2dt \right)^{\frac{1}{2}}\right]\leq \dfrac{e^{\beta T}}{2}E\left[ \sup_{0\leq t\leq T}\vert \bar{Y}(t) \vert^2+\int_0^T\vert \bar{Z}(t) \vert dt \right]<\infty,\\
&E\left[\left( \int_0^T e^{2\beta t}\vert \bar{Y}^{(l)}(t) \vert^2\vert \bar{S}_{ij}^{(l)}(t) \vert^2 d\left[Q_{ij}\right](t) \right)^{\frac{1}{2}}\right]\\
&\leq \dfrac{e^{\beta T}}{2}E\left[ \sup_{0\leq t\leq T}\vert \bar{Y}^{(l)}(t) \vert^2+\int_0^T\vert \bar{S}^{(l)}_{ij}(t) \vert^2d\left[Q_{ij}\right](t) \right]<\infty.
\end{align*}
Hence $ \int_0^t e^{\beta s}\bar{Y}(s)^\intercal\bar{Z}(s)dW(s) $ and $ \int_0^te^{\beta s}\sum_{l=1}^n\sum_{i,j=1}^d\bar{Y}^{(l)}(s)\bar{S}_{ij}^{(l)}(s)dQ_{ij}(s) $ are true martingales by the Burkholder-Davis-Gundy inequality. Taking expectation in \eqref{BSDEito}, we get
\begin{equation}\label{BSDEtosub}
\begin{array}{ll}
E\vert \bar{Y}(0) \vert^2 + E\bigg\lbrace\displaystyle\int_0^T e^{\beta t}\bigg[ \left(\beta\vert \bar{Y}(t) \vert^2+\vert \bar{Z}(t) \vert^2\right)dt+\sum_{l=1}^n\sum_{i,j=1}^d\vert \bar{S}^{(l)}_{ij}(t) \vert^2d\left[Q_{ij}\right](t) \bigg]\bigg\rbrace\\
= \displaystyle2E\left[\int_0^Te^{\beta t}\bar{Y}(t)^\intercal\bar{f}(t)dt \right] \leq 2C_fE\left[ \int_0^Te^{\beta t}\vert \bar{Y}(t) \vert\left( \vert \bar{U}(t)\vert+\vert\bar{V}(t) \vert \right)dt \right]\\
\leq \displaystyle 4C_f^2E\bigg[\int_0^T e^{\beta t}\vert \bar{Y}(t) \vert^2dt \bigg]+\dfrac{1}{2}E\bigg[ \int_0^T e^{\beta t}\left( \vert \bar{U}(t) \vert^2+\vert \bar{V}(t) \vert^2 \right) dt \bigg].
\end{array}
\end{equation}
Take $ \beta=1+4C_f^2 $ and substitute into \eqref{BSDEtosub}, we have
\begin{align*}
&E\bigg[ \int_0^T e^{\beta t}\left( \vert \bar{Y}(t) \vert^2+\vert \bar{Z}(t) \vert^2 \right)dt+\int_0^Te^{\beta t}\sum_{l=1}^n\sum_{i,j=1}^d\vert \bar{S}^{(l)}_{ij}(t)\vert^2d[Q_{ij}](t) \bigg]\\
&\leq \dfrac{1}{2}E\bigg[ \int_0^T e^{\beta t}\left( \vert \bar{U}(t) \vert^2+\vert \bar{V}(t) \vert^2 \right)dt \bigg] \\
&\leq \dfrac{1}{2}E\bigg[ \int_0^T e^{\beta t}\left( \vert \bar{U}(t) \vert^2+\vert \bar{V}(t) \vert^2 \right)dt \bigg]+\dfrac{1}{2}E\bigg[ \int_0^Te^{\beta t}\sum_{l=1}^n\sum_{i,j=1}^d\vert \bar{\Gamma}^{(l)}_{ij}(t) \vert^2d[Q_{ij}](t) \bigg].
\end{align*}
Notice that $ L^2(W,[0,T]) $ and $ L^2(Q,[0,T]) $ are Hilbert spaces and therefore the space $ \mathbb{S}^2([0,T])\times L^2(W,[0,T])\times L^2(Q,[0,T]) $ endowed with the norm 
\begin{align*}
\| (Y,Z,S) \|_{\beta}=\left\lbrace E\bigg[ \int_0^T e^{\beta t}\left( \vert \bar{Y}(t) \vert^2+\vert \bar{Z}(t) \vert^2 \right)dt+\int_0^Te^{\beta t}\sum_{l=1}^n\sum_{i,j=1}^d\vert\bar{S}^{(l)}_{ij}(t)\vert^2 d[Q_{ij}](t) \bigg] \right\rbrace^{\frac{1}{2}}
\end{align*}
is a Banach space. We conclude that $ \Phi $ admits a unique fixed point which is the solution to the BSDE \eqref{BSDE}.
\end{proof}

\bibliographystyle{plain-annote}
\bibliography{bib_ref}

\end{document}